\documentclass[12pt, reqno]{article}

\usepackage{fullpage}
\usepackage{enumerate}
\usepackage{setspace}
\usepackage{listings}
\usepackage{amsmath,amsfonts,amssymb,graphics,amsthm, comment}
\usepackage{hyperref}
\hypersetup{
    colorlinks=true,
    linkcolor=blue,
    citecolor=red,
    urlcolor=blue,
    pdfborder={0 0 0}
}            

\usepackage{graphicx}
\usepackage{xcolor}
\usepackage{wrapfig} 
\usepackage{xcolor}

\usepackage[font=sf, labelfont={sf,bf}, margin=1cm]{caption}

\usepackage{cleveref}
  \crefname{theorem}{Theorem}{Theorems}
  \crefname{thm}{Theorem}{Theorems}
  \crefname{lemma}{Lemma}{Lemmas}
  \crefname{lem}{Lemma}{Lemmas}
  \crefname{remark}{Remark}{Remarks}
  \crefname{prop}{Proposition}{Propositions}
  \crefname{proposition}{Proposition}{Propositions}
  \crefname{problem}{Problem}{Problems}
\crefname{notation}{Notation}{Notations}
\crefname{claim}{Claim}{Claims}
  \crefname{defn}{Definition}{Definitions}
  \crefname{corollary}{Corollary}{Corollaries}
  \crefname{section}{Section}{Sections}
  \crefname{figure}{Figure}{Figures}
  \crefname{exercise}{Exercise}{Exercises}
    \crefname{assumption}{Assumption}{Assumptions}

\newtheorem{thm}{Theorem}[section]

\newtheorem{conjecture}[thm]{Conjecture}

\newtheorem{lemma}[thm]{Lemma}

\newtheorem{corollary}[thm]{Corollary}
\newtheorem{prop}[thm]{Proposition}

\newtheorem{proposition}[thm]{Proposition}

\numberwithin{equation}{section}

\theoremstyle{definition}
\newtheorem{remark}[thm]{Remark}

\def\cP{\mathcal{P}}

\def\cM{\mathcal{M}}
\def\cL{\mathcal{L}}

\def\cH{\mathcal{H}}
\def\cG{\mathcal{G}}
\def\cF{\mathcal{F}}
\def\cE{\mathcal{E}}
\def\cD{\mathcal{D}}
\def\cC{\mathcal{C}}
\def\cB{\mathcal{B}}

\usepackage{wasysym}
\def \bP {\mathbf P}

\def \ve {\varepsilon}

\def\P{\mathbb{P}}
\def\E{\mathbb{E}}
\def\C{\mathbb{C}}
\def\R{\mathbb{R}}
\def\Z{\mathbb{Z}}
\def\N{\mathbb{N}}
\def\D{\mathbb{D}}

\def\R{\mathbb{R}}

\def  \p- {p\textunderscore}

\def\eps{\varepsilon}

\renewcommand{\i}{{\bf i}}

\usepackage{extarrows}

\title{Double dimers on planar hyperbolic graphs via circle packings}
\author{Gourab Ray \thanks{University of Victoria. Research partially supported by NSERC 50311-57400. Address: DTB A440, University of Victoria, Victoria, Canada, V9B 0Z2. Email: gourabray@uvic.ca}}

\begin{document}
\maketitle
\abstract{In this article we study the double dimer model on hyperbolic Temperleyan graphs via circle packings. We prove that on such graphs, the weak limit of the dimer model exists if and only if the removed black vertex from the boundary of an exhaustion converges to a point on the unit circle in the circle packing representation of the graph. One of our main results is that for such measures,  the double dimer model has no bi-infinite path almost surely.

Along the way we prove that in the nonamenable setup, the height function of the dimer model has double exponential tail and faces of height larger than $k$ do not percolate for large enough $k$. The proof uses the connection between winding of uniform spanning trees and dimer heights, the notion of stationary random graphs, and the boundary theory of random walk on circle packings.}
\section{Introduction and main results}
The double dimer model, which is simply a superimposition of two independent samples from a measure supported on perfect matchings, has a long history \cite{kenyon2014conformal,DDD19,BD19,L17} and is considered an important problem in the field. Despite great progress in the planar setting of lattices and other flat geometries, dimer model on graphs of other general geometries have not seen much progress. On the contrary, for Bernoulli percolation, fascinating progress has been made on percolation in hyperbolic geometry, creating some of the most beautiful and elegant mathematics in the past two decades, see \cite{LyonsPeres,bperc96} for an overview of  this topic.
In this article, we study the double dimer model on planar hyperbolic graphs, with the main goal to prove that bi-infinite paths do not exist in them.

Let $\Gamma$ be a planar, one ended, transient, random rooted triangulation with degree uniformly bounded by $\Delta$.   For precise definitions, we refer to \Cref{sec:stationary}. This includes regular triangulations with degree at least 7 as special cases. Unlike percolation, we need to set things up more precisely to study the dimer model on such graphs. To that end, we now briefly describe how to obtain the Temperleyan version of $\Gamma$ (a more detailed exposition is in \Cref{sec:height}).
Let $\Gamma^\dagger$ denote the dual graph of $\Gamma$. Let $G$ denote the \textbf{Temperleyan graph} formed by superimposing $\Gamma$ and $\Gamma^\dagger$ and introducing a new vertex at the intersection points of each edge and their dual. Color the vertices of $\Gamma, \Gamma^\dagger$ black and the new vertices introduced white. Note that $G$ is in fact a bipartite graph, which is in fact a quadrangulation, and the colorings of its vertices thus introduced partitions the vertex set into its partite classes. (see \Cref{fig:temp}). Recall that a
\textbf{dimer covering} or a \textbf{perfect matching} is a collection of edges in $G$ such that every black vertex is incident to exactly one white vertex.

 We now define the `uniform measure' of dimer coverings on $G$. 
 Since $G$ is infinite, we define it through the standard technique of creating an exhaustion of finite approximations of $G$ and then taking limit. Take a planar, proper embedding of $(\Gamma, \rho)$. Let $\Gamma^\dagger_1 \subseteq \Gamma^\dagger_2 \subseteq \dots$  denote an exhaustion of $\Gamma^\dagger$  such that the union $\Gamma^\dagger_i$ and the faces of $\Gamma^\dagger_i$ form a simply connected subset of the plane. Let $\Gamma_i$ denote the planar dual of $\Gamma^\dagger_i$. Let $\hat G_i$ denote the graph obtained by superimposing $\Gamma_i, \Gamma_i^\dagger$, introducing a new white vertex  as above.   \footnote{It is somewhat of a convention to use the dual graph for the wired boundary graph and the primal for the free boundary graph. However since all the random walk estimates are in the wired graph, we make that the primal graph here, which hopefully will not be a cause for confusion.}
 Now remove the vertex in $\Gamma_i$ corresponding to the outer face, and remove remove a black vertex from the boundary of $\hat G_i$ and call the resulting graph $G_i$, see \Cref{fig:temp} (this last step is necessary to ensure there is a dimer cover). Call the black vertex removed $\mathfrak b_i$. Let $\mu_i$ denote the uniform measure on the set of dimer covers of $G_i$. We note that if we superimpose two dimer covers $m,m'$ then $m\Delta m'$ consists of finite cycles with alternating edges from $m $ and $m'$.

 It turns out that in this setup, the weak limit of the dimer covers may not exist, and is dependant heavily on the choice of the sequence of removed vertices $\mathfrak b_i$. However, we shall prove that there are  judicious choice of the sequences $\mathfrak b_i$ for which the limit does exist. The details of such a choice can be found in \Cref{sec:infinite_limit}, however for now we quickly mention that the way to obtain this sequence is to circle pack $\Gamma$ in a unit disc $\D$ and then ensure that $\mathfrak b_i$ converges to a point in $\partial \D$. We prove the following theorem about weak limits.
 
 \begin{thm}\label{thm:main0}
 Suppose $(\Gamma, \rho)$ is one ended, bounded degree, transient triangulation, and let $G$ be the Temperleyan graph obtained as above.
 Suppose $\mathfrak b_i$ is chosen as above so that $\mathfrak b_i$ converges to a point $x$ in $\partial \D$ for a circle packing embedding of $\Gamma$ and let $\mu_i^x$ be the corresponding uniform dimer measure on $G_i$. Then  the sequence $(\mu^{x}_i)_{i \ge 1}$ converges weakly to $\mu^x$, a probability measure on dimer covers of $G$. 
 \end{thm}
The following theorem about double dimers can be seen as one of the main contribution of this article. For this theorem we need the additional assumption that $(\Gamma, \rho)$ random rooted graph which is stationary with respect to simple random walk (see \Cref{sec:stationary} for details). Furthermore, we need $\Gamma$ to be nonamenable almost surely (see \eqref{eq:nonamenable}).
 \begin{thm}\label{thm:main}
 Suppose $(\Gamma, \rho)$ is a reversible random rooted graph which is a one ended, bounded degree, nonamenable  triangulation almost surely, and let $G$ be the Temperleyan graph obtained as above. Condition on $(\Gamma, \rho)$ and circle pack it arbitrarily in the unit disc. There exists a dense set of points $\chi \subset \partial \D$ such that the following is true. Let $M,M'$ denote two independent dimer covers distributed as $\mu^x$ where $x \in \chi$. Then $M \Delta M'$ has no bi-infinite path almost surely.
 \end{thm}
 The set $\chi$ we take is a set of exit measure $1$ of a simple random walk seen from $\rho$, it is known that any such set must be dense in $\partial \D$, see \Cref{thm:RW_circle_packing}. We believe that the theorem is true for $\chi = \partial \D$, the additional restriction is an artefact of the ergodic theoretic components in the proof.

Along the way, we prove various estimates on single dimer model as well.
We shall heavily rely on the height function representation of dimers and  its connection to winding of uniform spanning trees which have been exploited in great depth in recent years \cite{BLR16,BLR_Riemann1,BLR_Riemann2,BHS,BLiu,ray2021quantitative}. We refer to \Cref{sec:height} for details on how to construct the height function, but for now let us mention that it is a function $h:F(G) \to \R$ and is a unique representation of each dimer cover. One by product of our methods is that the height function converges to an infinite volume Gibbs measure, or in other words, it is \textbf{localized}. The following theorem tells us further that the height function is extremely flat in the sense that the tail is double exponential, uniformly over the faces of the graph. We denote by $h^x$ the height function obtained by sampling a dimer cover distributed as $\mu^x$.
\begin{thm}\label{thm:height}
Let $\Gamma$ be a bounded degree, transient, one ended triangulation.
Let $x \in \partial \D$. Then the weak limit of the height function  $(h^x_i(f))_{f \in G_i}$ corresponding to $\mu^x_i$ exists and is non-trivial in the sense that $|h^x(f)|  < \infty$ almost surely for all $f \in G$. 

Furthermore, if $\Gamma$ is assumed to be nonamenable, then there exists $C,c>0$ such that for all $k>0$ and $f \in F(G)$
$$\P(|h^x(f)| > k) \le C\exp(-e^{ck}).$$
\end{thm}
From the methods of the proof, it is not too hard to see that the bound above is sharp, although we do not prove this. It is also not too hard to see that without the assumption of nonamenability such uniform bound cannot hold: for example one can find large regions in the triangulation which look like the (six-regular) triangular lattice, and there the methods of our proof will yield an exponential lower bound on the height, namely, the winding of a uniform spanning tree branch started at the center of such a  region has exponential lower bound on the tail. We do not pursue the details of these claims for the sake of brevity.

Analogues can be drawn between \Cref{thm:height} and the results in \cite{PWY_expander,PWY_tree} for uniform lipschitz functions. 
Note that \Cref{thm:height} has an immediate consequence: the maximum of the height function in a graph distance ball of radius $n$ (call it $B_\Gamma(\rho, n)$) around $\rho$ grows at most like $\log \log |B_\Gamma(\rho, n)|$. We quickly mention here that in an expander of size $n$ it is known that uniform Lipschitz functions have maximal height of order $\log \log n$ \cite{PWY_expander,BHM_00}.  
It is worth mentioning also that it is not even known whether uniform lipschitz functions (with 0 boundary condition) is localized on $(\Gamma, \rho)$. All these results are in sharp contrast with those in planar lattices ($\Z^2$ ,hexagonal or triangular lattices depending on the model) where it is known that the height function is delocalized (i.e. variance of the height at $f$ diverges to $\infty$) both in the dimer case (for an appropriate variation of the lattice) \cite{KenyonGFF,BLR16} and in the uniform Lipschitz or homomorphism case, the latter has seen a huge surge of interest in the past few years \cite{CPST_delocalisation,duminil2022logarithmic,karrila2023logarithmic,glazman2021uniform,duminil2020delocalization}.

Despite being localized, one can ask other interesting questions about percolation of the height clusters. In particular, one may ask, does a localized height function model  with mean 0 sees a percolation of  the sites with height strictly positive? In the case of the Gaussian free field in Euclidean lattices, this has received a lot of attention in recent years (see \cite{rodriguez2015level} for an overview). In an integer valued localized model with low temperature and zero boundary conditions (like integer valued GFF) a simple Peierl's argument will tell us that the sites with height not 0 will not percolate. Nevertheless, it is an interesting question whether at high temperature we have percolation of strictly positive sites or not (it is easy to see that non-zero heights do percolate at high enough temperature for such models via a Dobrushin type argument). This is one of the motivations to ask if in the double dimer model (with zero height boundary condition, as is the case here), does one see an infinite component of height at least 1, although admittedly there is no temperature parameter to tune here. Our \Cref{thm:main} tells us that this is not the case for dimer model, it is mostly flat. We also prove the following theorem for the single dimer model.

\begin{thm}\label{thm:height_perc}
Let $\Gamma, \mu^x, (h^x(f))_{f \in G}$ be as in \Cref{thm:height_perc}.
 Then there exists a $k_0$ such that for all $k>k_0$, $\cH_k:= \{f:|h^x(f)| >k\}$ does nor contain any infinite component.
\end{thm}
We do not know if \Cref{thm:height_perc} is true in the setup of \Cref{thm:main0}, although we are tempted to conjecture that it is true. In fact our methods will show this to be true once we just have thin enough exponential tail on the height uniform over the faces of the graph,  depending only on $\Delta$. Note that the latter condition is far weaker than \Cref{thm:height}.
Questions analogous to \Cref{thm:height} for other height function models, particularly integer valued ones, could be of independent interest.

Finally we remark that we believe  bounded degree assumption to be not needed in any of our theorems, as long as the triangulation remains CP hyperbolic. However, for the sake of brevity, we do not pursue this direction of generalization.

\subsection{Outline of the proof} Broadly, our main tool to address the problem is to use the connection between height functions and winding of uniform spanning trees, which have been exploited heavily in the past few years \cite{BLR16,BLR_Riemann1,BLR_Riemann2,BLiu,BHS,ray2021quantitative}. The setup and this connection is explained in \Cref{sec:height}. The bijection holds between the dimer model on the wired-free uniform spanning tree pair on a  Temperleyan version $G$ of $\Gamma$, where the free tree is oriented towards a boundary vertex $\mathfrak b$. In the weak limit, this yields a dual wired-free uniform spanning tree pair, with the free tree having an oriented end.
Although it is well known that the weak limit of the dual free-wired uniform spanning tree does exist (see \cite{LyonsPeres}, Chapter 10 and references therein), a priori, from this description, it is not even clear if the weak limit exists for the dimer model as the oriented ray of the free tree might not converge locally. 

The way to understand this is to use the circle packing embedding of $(\Gamma, \rho)$ which goes back to the work of Koebe, Andreev and Thurston \cite{Th78,K36}. It is known since the work of He and Schramm \cite{HeSc} that a transient, one ended triangulation $\Gamma$ can be circle packed in the unit disc and the embedding is unique up to Mobius transforms of the disc \cite{Schramm91}. With this embedding, we prove that for every boundary point $x$  on the unit circle $\partial \D$, there exists an end of the free spanning tree converging towards $x$ almost surely. Furthermore, it is also known that every component of the wired spanning tree  is one-ended and converges to a boundary point. Thus if we take an exhaustion with the vertex $\mathfrak b$ converging towards $x$, we obtain a weak limit of the dimer model. Furthermore, in such a limit, the height function at $v$ (which is normalized to be 0 at $x$ in an appropriate sense) can be obtained by computing the winding of the wired spanning tree branch all the way to $\xi \in \partial \D$ and then adding the angle obtained by going along $\partial \D$ anticlockwise  started at $\xi$ and ending at $x$. With this setup, we analyze the simple random walk in \Cref{sec:height_percolation} to deduce the double exponential tail of the winding of the wired tree branch. We use this strong bound on the tail to deduce \Cref{thm:height} as well. The same argument will also show that the result of \Cref{thm:height} is also true for the difference of heights of two i.i.d.\ dimer covers sampled from the same limit measure. We note that this part of the argument do not require $(\Gamma, \rho)$ to be stationary, and works for any deterministic nonamenable triangulation.

Note that we assumed the graph $(\Gamma, \rho)$ to be reversible with respect to simple random walk, and in fact this graph decorated with the unoriented wired and the free uniform spanning tree also remains reversible. Nevertheless, adding the orientation to the free tree makes the graph lose its reversibility, and brings us to the setup of nonunimodular stationary graphs. One prime example of such graphs is a 3-regular tree oriented towards one of its ends, or the grandfather graph (see \cite{AL07}). Although we lose reversibility, we can make the model stationary with respect to random walk by choosing the free tree to be oriented towards a boundary point chosen using the exit measure of an independent simple random walk started at $\rho$.

Now we come to the crux of the argument for  \Cref{thm:main}. There are two cases: either there are infinitely many bi-infinite path or there are finitely many. The infinitely many bi-infinite paths case is quickly ruled out as this would yield infinite components of height for any $k \in \N$. Indeed, the height function along the faces adjacent to such a bi-infinite path is constant and hence form an infinite height component. To rule out finitely many, note that uniformly choosing a bi-infinite path in this case yields a stationary random marked graph as well. By stationarity, a simple random walk will cross this path infinitely often. Now if we run several independent random walks, using the estimates from \Cref{sec:height_percolation}, we can conclude that the height functions along the random walks  become decorrelated. On the other hand, since the walks cross the \emph{same} bi-infinite path infinitely often, some correlation remain along the walks no matter how long we run the walks. We can exploit this to arrive at a contradiction.

\paragraph{Acknowledgement:} We are thankful to Tom Hutchcroft, Benoit Laslier and Zhongyang Li for several useful discussion. We specially thank Benoit for his contributions during the initial parts of the project. While preparing this manuscript, a paper \cite{Li_DD} was independently posted to the arXiv by Zhongyang Li containing some results which partially overlap the results in this paper. However the techniques used by Li appear to be substantially different from ours. 
 
\section{Background}\label{sec:background}
\subsection{Stationary random graphs}\label{sec:stationary}

Let $\cG^{\bullet}_m$ denote the space of all locally finite graphs $G = (V,E)$ with a marked root vertex and a function $m:V\cup E \to \cM$ where $\cM$ is a countable space, viewed up to graph isomorphisms which preserve the root and the marks. This space can be endowed with the \emph{local topology} which turns this into a Borel space. A random rooted marked graph is a probability measure on $\cG^{\bullet}_m$. Let $(G, \rho, m)$ be a random rooted marked graph. We say $(G, \rho, m)$ is stationary if $(G, \rho, m) = (G, X_1, m)$ in distribution, where $X_1$ is the vertex obtained by running a simple random walk for one step started at $\rho$. We can similarly define the space of all bi-rooted, locally finite, marked graphs $\cG^{\bullet \bullet}_m$. A random rooted marked graph $(G, \rho, m)$ is reversible if $(G, \rho, X_1, m)$ is equal in distribution to $(G,X_1, \rho, m)$. Clearly a reversible random rooted graph is also stationary. Nevertheless, a stationary random rooted graph may not be reversible: take a three regular tree with marks being orientation, and all the edges are oriented towards a fixed end of the tree.

It is well known \cite{AL07} that reversible random rooted graphs are related to unimodular random graphs: simple inverse degree biasing a reversible random graphs yield a unimodular random graphs. However since the graphs we consider with the orientation of the free uniform spanning tree is not reversible, we cannot use the plethora of results for unimodular marked nonamenable graphs, and indeed many of the results are actually false in such setups. 

A \emph{planar map} is a proper embedding of a graph into a subset of the plane viewed up to orientation preserving homeomorphisms of the plane. A face of a map is the connected components of the complement of the embedding. A planar map is a \textbf{triangulation} if every face is a triangle, that is if every face is bounded by exactly three edges. Note that a consequence of the definition is that the triangulation has to be one ended.

\subsection{Winding of curves}\label{sec:winding}
The goal of this section is to recall several notions of windings of
 curves drawn in the plane, which we use in this paper. We refer to \cite{BLR16} for a more detailed exposition. A self-avoiding (or simple) curve
 in $\C$ is an injective continuous map $\gamma:[0,T] \mapsto \C$ for some $T \in
[0,\infty]$.

\medskip The \textbf{topological winding} of such a curve $\gamma$ around a point $p \notin \gamma[0,T]$ is defined as follows. We first
write
\begin{equation}
 \gamma(t)  - p = r(t)e^{i \theta(t)},
\end{equation}
where the function $\theta: [0,\infty) \to \R$ is taken to
be continuous, which means that it is unique modulo a global additive constant multiple of $2\pi$. We define the
winding for an interval of time $[s,t]$, denoted
$W(\gamma[s,t],p)$, to be
$$
W(\gamma[s,t],p) = \theta(t) - \theta(s)
$$ (note that this is uniquely defined).
Notice that if the curve has a derivative at an endpoint of $\gamma$, we can take $p$ to
be this endpoint by
defining
$$W(\gamma[0,t],\gamma(0)) := \theta(t)- \lim_{s \to 0} \theta(s)$$
and
similarly
$$W(\gamma[s,T],\gamma(T)) := \lim_{t
\to T} \theta(t)-\theta(s).
$$

With this definition, winding
is
additive: for any $0 \le s \le t \le T$
\begin{equation}
 W(\gamma[0,t],p) = W(\gamma[0,s],p) + W(\gamma[s,t],p).
\end{equation}

The notion of \textbf{intrinsic winding} we describe
now, also discussed in \cite{BLR16}, is perhaps a more natural definition of windings of curves. This notion is
the continuous analogue of the discrete winding of non backtracking paths in $\Z^2$
which can be defined just by the number of anticlockwise turns minus the number
of clockwise turns. Notice that we do not need to specify a reference point with
respect to which we calculate the winding, hence our name ``intrinsic" for this
notion.


{In this paper we will call  a curve $\gamma$ smooth if the map $\gamma$ is continuously
differentiable,
and for all $ t, \,\gamma'(t) \neq 0$.} We write
$\gamma'(t) = r_\i(t)e^{i \theta_\i(t)}$ where again
$\theta_{\i} : [0,\infty) \to \R$ is taken to be continuous. Then
define the intrinsic winding in the interval $[s,t]$ to be
\begin{equation}\label{E:windingsmooth}
W_{\i}(\gamma[s,t]) :
= \theta_\i(t)- \theta_\i(s).
\end{equation}
The total intrinsic winding is again defined to be
$\lim_{t \to T} W_{\i}(\gamma,[0,t])$ provided this limit exists. Note that this
definition does not depend on the parametrisation of $\gamma$ (except for the assumption of
non-zero derivative). The following topological lemma
from \cite{BLR16} connects the intrinsic and topological windings for smooth curves.


\begin{lemma}[Lemma 2.1 in \cite{BLR16}]\label{lem:int_to_top}
 Let $\gamma[0,1]$ be a smooth {simple} curve in $\C$ then,
$$
 W_{\i}(\gamma[0,1]) = W(\gamma[0,1],\gamma(1)) + W(\gamma[0,1],\gamma(0)).
$$
\end{lemma}
{Note that the right hand side in the above expression makes sense as soon as the endpoints of $\gamma$ are smooth. From now on we therefore extend the definition of intrinsic winding to all such curves using \cref{lem:int_to_top}. Since the topological winding is clearly a continuous function on the curve away from the point $p$, this extension is consistent with any regularisation procedure.}
We also recall the following deformation lemma from \cite{BLR16} (see Remark 2.5).
\begin{lemma}\label{lem:winding_change_conformal}
Let $D$ be a domain and $\gamma \subset \bar D$ a simple smooth curve (or piecewise smooth with smooth endpoints). Let $\varphi$ be a conformal map on $D$ and let $\arg_{\varphi'(D)}$ be any realisation of argument on $\varphi'(D)$. Then
\begin{equation*}
W_{\i}( \varphi( \gamma ) )-  W_{\i}(\gamma) = \arg_{\varphi'(D)}(\varphi'(\gamma(1) ) ) - \arg_{\varphi'(D)}( \varphi'(\gamma(0) ) ).
\end{equation*}
\end{lemma}

\subsection{Temperleyan graphs, height function and windings}\label{sec:height}
In this section, we recall the existing theory of height function and winding and integrate them into the existing theory of embedding of hyperbolic planar graphs.

Let us first recall the construction of Temperleyan graphs which goes back to Temperley-Fisher \cite{TF_61} and we will consider a version exploited in \cite{KPWtemperley,BLR16,BLR_Riemann1}. Let $\Gamma^\dagger$  be a planar graph embedded properly in the plane and assume that the complement of the embedding is a disjoint union of topological discs . Let $\Gamma$ be its dual graph also embedded properly so that $\Gamma$ and $\Gamma^\dagger$ cross at a single point. We assume both $\Gamma, \Gamma^\dagger$ do not contain any self-loops.
Let $\hat G$ be a finite graph obtained by superimposing a graph $\Gamma$ and its dual as described in \Cref{fig:temp}: introduce a white vertex on every point where an edge and its dual crosses. This yields a bipartite planar graph where the black vertices are those in $\Gamma \cup  \Gamma^\dagger$. Let $G$ be the graph obtained by removing two black vertices from $\hat G$, one corresponding to the outer face and another on the boundary cycle of $\hat G$. Call the former vertex $\partial$ (expanded into a circular arc on the left side of \Cref{fig:temp}) and the latter $\mathfrak b$.  
\begin{figure}[h]
\centering
\includegraphics[scale = 0.53]{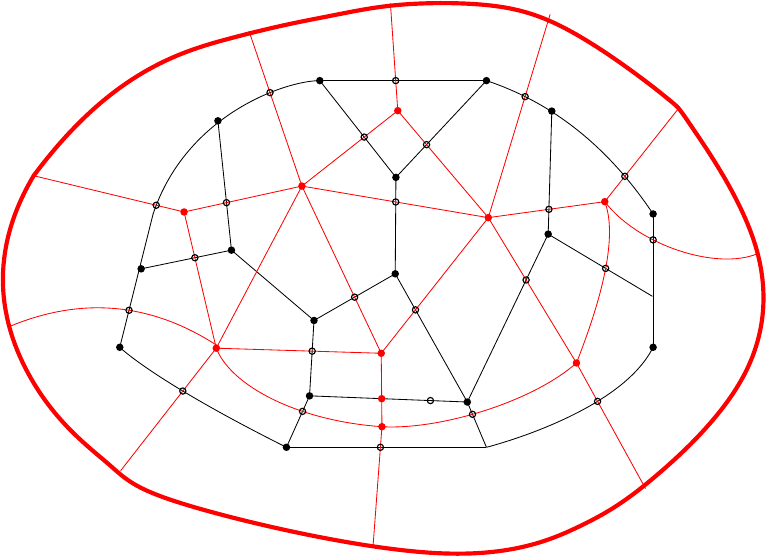}\qquad
\includegraphics[scale  = 0.53]{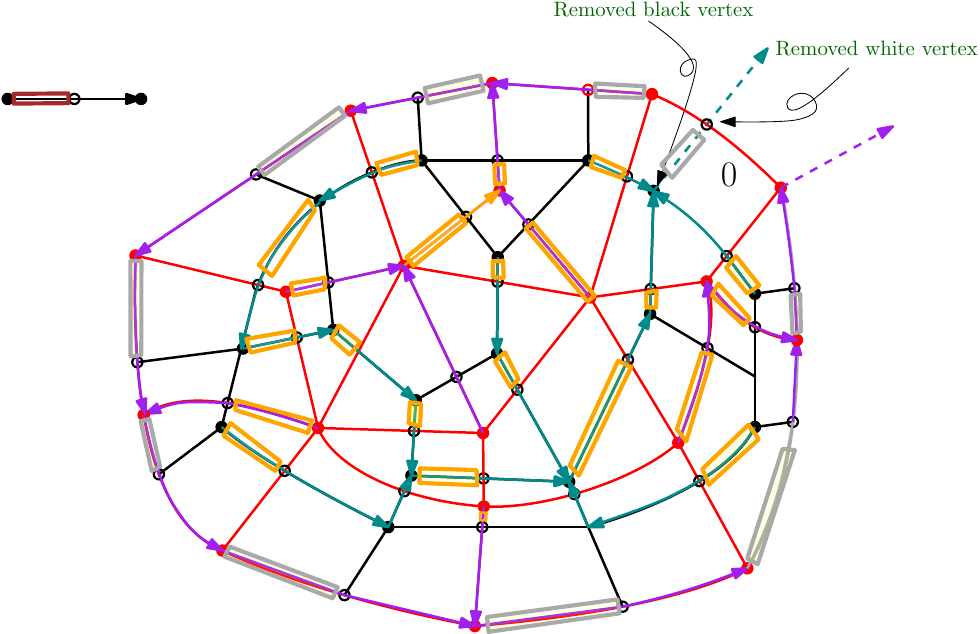}
\caption{Left: The black graph is $\Gamma^\dagger$ and the red graph is $\Gamma$, the outer vertex is expanded into a circular arc for convenience. Right: The extended graph $\bar G$ is drawn. The purple tree is the wired spanning forest, where the circular arc corresponding to the outer vertex is artificially added. The green tree is the free spanning forest oriented towards the removed vertex. The dotted arrows can be artifically expanded into an infinite graph to facilitate the understanding of the height function and winding correspondence. The face $\mathfrak f$ with height 0 is marked. The dimers artificially added on the boundary cycle is marked gray. The removed black vertex is $\mathfrak b$ and the removed white vertex is $\mathfrak w$.}\label{fig:temp}
\end{figure}

It is a standard fact that the graph $G$ (see \cite{BLR16,KPWtemperley}) admits a dimer cover.  Furthermore, there is a simple correspondence between dimer coverings of $G$ and uniform spanning trees of $\Gamma, \Gamma^\dagger$ which we now describe. Let $\bf m$ denote a dimer covering of $G$. Let $(b,w)$ be an edge covered by a dimer in $G$ with $b$ being the black vertex and $w$ the white vertex. Simply draw an arrow as in \Cref{fig:temp}. It is known (see \cite{KPWtemperley,BLR_Riemann1}) that this yields two spanning trees, one of $\Gamma$ and another of $\Gamma^\dagger$. Furthermore, the spanning tree of $\Gamma^\dagger$ is oriented towards $\mathfrak b$ and that of $\Gamma$ is oriented towards $\partial$.

We will use the following extension of the trees \emph{without orientation} to infinite volume. It is a well established fact that wired uniform spanning trees (and) consequently the free uniform spanning forests have a weak limit (See \cite{LyonsPeres}, Chapter 10 and references therein). Let $(\Gamma_i, \Gamma_i^\dagger)$ be an exhaustion of $\Gamma, \Gamma^\dagger$ as described in the introduction. Let $T_i, T_i^\dagger$ be the dual wired-free pair of \textbf{unoriented} uniform spanning trees of $\Gamma, \Gamma^\dagger$ respectively. 
\begin{thm}[\cite{LyonsPeres}, Chapter 10]\label{thm:UST_limit}
The weak limit of $T_i, T_i^\dagger$ exists, call them $T, T^\dagger$ respectively. Furthermore, $T$ has infinitely many infinite components, every component of $T$ is one-ended almost surely and $T^\dagger$ is connected almost surely.
\end{thm}

We now define a function real valued function $h$ from the faces of $G$ to $\R$ which is called the height function, and it represents the dimer cover in a convenient way. Furthermore, the height function  is intimately related to the winding of the branches of the uniform spanning trees of $\Gamma, \Gamma^\dagger $. We follow the exposition of \cite{BLR_Riemann1}. However, to properly define the boundary conditions, it is convenient to modify $G$ slightly which we describe now.

 We expand the outer vertex $\partial$ of $\Gamma$ into a cycle $C_\partial$ of length equal to the degree of $\partial$. Then we introduce a white vertex at the center of every edge of $C_\partial$, and join every vertex in the boundary cycle of $\Gamma^\dagger$  to a white vertex as shown in \Cref{fig:temp}. Now we add a dimer alternately joining vertices of $C_\partial$ and the white vertices, except for the white vertex $\mathfrak w$ adjacent to $\mathfrak b$. Let $\mathfrak e$ denote the edge joining $\mathfrak b$ and $\mathfrak w$ and we further add a dimer to $\mathfrak e$ (these artifically added dimers are colored gray in \Cref{fig:temp}). Note that  the uniform dimer cover in the rest of $G$ has the same law as the dimer cover of this extended graph with these boundary conditions.  Call this extended graphs $\bar {\Gamma}, \bar { \Gamma^\dagger}, \bar G$.

An advantage of this boundary condition is that the Temperleyan bijection for any dimer cover $\sf m$ of $G$ takes gives us a particularly nice pair of spanning trees $(T_{\sf m}, T_{\sf m}^\dagger)$ of $\bar \Gamma, \bar \Gamma^\dagger$. For $\Gamma$, we get a spanning tree which includes all but one edge of $C_\partial$, and this tree is oriented along this cycle as shown by the purple tree in \Cref{fig:temp}. We will call this tree $T_{\sf m}$ the \textbf{wired spanning tree}. The spanning tree $T^\dagger_{\sf m}$ of $\Gamma^\dagger$ is oriented towards ${\mathfrak b}$. We will call this the \textbf{free spanning tree}. Ignoring the edges in $C_\partial$, we see that  $(T_{\sf m}, T_{\sf m}^\dagger)$ are really spanning trees of $\Gamma, \Gamma^\dagger$, extended in a nice way because of the imposed boundary conditions.

\begin{figure}[h]
\centering
\includegraphics[scale = 0.5]{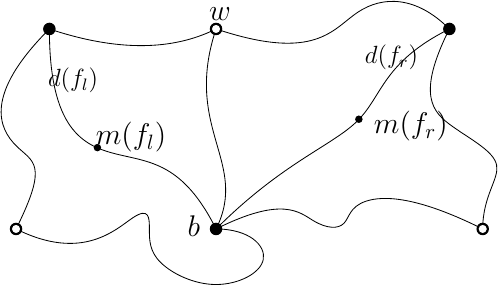} \qquad \qquad \qquad \qquad
\includegraphics[scale = 0.5]{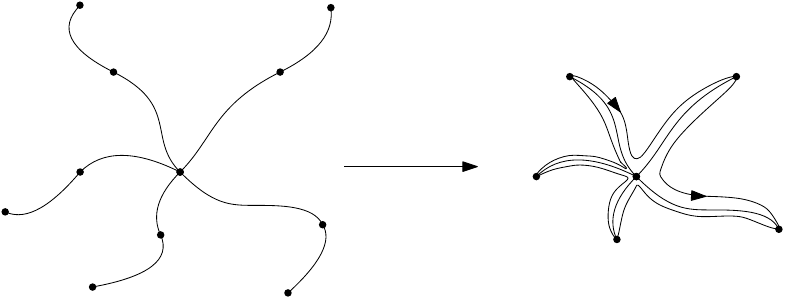}
\caption{Left: Midpoints drawn at the diagonal. Right: Proof that $\omega_{\text{ref}}$ is a valid reference flow.}\label{fig:ref}
\end{figure}

Let $F(\bar G)$ denote the faces of $\bar G$ except the outer face, and note that every element of $F(\bar G)$ is a quadrangle (i.e. bounded by four edges) with diagonally opposite vertices being of the same color. 
Draw a diagonal joining the black vertices in every face in $F(G)$ and let $m(f)$ denote the midpoint of this diagonal on face $f$ (see \Cref{fig:ref}). We define the reference flow as follows. Let $w,b$ be two adjacent white and black vertices and let $f_l, f_r$ denote the faces of $\bar  G$ to the left and right of the oriented edge $(w,b)$. Here we only consider edges with at least one vertex in the unextended graph $G$, hence the edges along the boundary cycle $C_\partial$ are excluded. Let $m(f_l), m(f_r)$ denote their midpoints. Define the reference flow to be 
 \begin{equation}
 \omega_{\text{ref}} (wb):=\frac1{2\pi}(W_{\i} ((m(f_r),b) \cup (b,m(f_l))) + \pi).
 \end{equation}
It is proved in \cite[Lemma 4.3]{BLR_Riemann1} that $\omega_{\text{ref}}$ is a valid reference flow where the flow out of every black vertex in $G$ is -1 and that out of every white  vertex in $G$ is $1$. Let $\omega_{\text{dim}}((w,b)) = 1 = -\omega_{\text{dim}}((b,w))$ if and only if $(w,b)$ is covered by a dimer. Then $\omega_{\text{dim}} - \omega_{\text{ref}}$ denotes a divergence free flow on the edges of $G$, and hence it allows us to define a height function on $F(\bar G)$, up to a global shift.  Let $\mathfrak f$ denote the face to the right of the edge $\mathfrak e$ and fix the height of $\mathfrak f$ to be $0$. This completes the description of the height function $h_{\bf m}: F(\bar G) \to \R$.



With this convention, to calculate the height function of any internal face $f$ of $\hat G$, we can compute the winding of the branch of either $T_{\sf m}$ or $T^\dagger_{\sf m}$ as follows. Let $v^\dagger_f$ and $v_f $ denote the vertices corresponding to $\Gamma$ and $\Gamma^\dagger$ respectively incident to $f$. Let $\gamma$ denote the branch of $T^\dagger_{\sf m}$ connecting $v^\dagger_{f}$ and $v^\dagger_{\mathfrak f}$. Now concatenate $(m(f), v^\dagger_f) , \gamma, (v^\dagger_{\mathfrak f},m(\mathfrak f))$ to define the path $\gamma^\dagger_{f}$. Similarly define $\gamma_f$.

\begin{lemma}[\cite{BLR_Riemann1}]\label{lem:height_winding_finite}
We have
$$
h(f) = W_{\i}(\gamma_{T_{\sf m}}(f) )= W_{\i}(\gamma_{T_{\sf m}^\dagger}(f))
$$
\end{lemma}

\begin{lemma}[Temperleyan bijection]
Let $G, \Gamma, \Gamma^\dagger$ be as above.
Applying the Temperleyan bijection to a uniform dimer covering $\sf M$ of $G$, we obtain a pair $(\vec T^\dagger_{\sf M}, \vec T_{\sf M})$ where the edges of $\vec T^\dagger_{\sf M}$ restricted to $\Gamma$ is a spanning arborescence of $\Gamma ^\dagger$ oriented towards $\mathfrak b$ and $\vec T_{\sf M}$ is a uniform  spanning arborescence of $\Gamma$ oriented towards $\partial$. 
\end{lemma}

We now explain how to interpret the weak limit of dimers in in the context of the weak limit of the spanning trees as in \cref{thm:UST_limit}.
To use the operation of Temperleyan bijection, we need to orient the trees towards one of the ends.
Once we can fix an end of each component of $T,T^\dagger$, we can revert the Temperleyan bijection to obtain a dimer configuration. Note that there is no ambiguity in the choice of the orientation of $T$. However the free tree $T^\dagger$ has various choices of orientation which in finite volume depends on the choice of the removed vertex $\mathfrak b$. Nevertheless, any two subsequential weak limits can be related to each other modulo a change in orientation dictated by the choice of the end of $T$. We expand on how to choose a sequence of exhaustion with $\mathfrak b$ judiciously so that there is a weak limit satisfying various invariance properties inherited by the graph decorated by the uniform spanning forests in \Cref{sec:embedding}. It will also be clear later that arbitrary sequence of choices of $\mathfrak b$ can result in the sequence of dimer measures not converging at all, a phenomenon which is absent in amenable setups like $\Z^2 $ or other planar regular lattices \cite{Kenyon_ci,BLR16}.

We finish this section by describing the celebrated Wilson's algorithm to sample the Uniform spanning tree. In order to do so, we need to describe the so called loop erased random walk (LERW) whose study was pioneered by Lawler (see \cite[Chapter 11]{Lawlerbook} for a detailed introduction). In a finite graph, let $\partial$ be a fixed vertex which we call the boundary. The LERW from $x$ to $\partial$ is sampled as follows. 

\begin{itemize}
\item Start a simple random walk from $x$ and continue until it hits $\partial$.
\item Erase the cycles chronologically. More precisely, having defined the walk $(x_0, \ldots, x_i)$, perform a simple random walk step from $x_i$ along the edge $(x_i,x_{i+1})$. If $x_{i+1} = x_j$ for some $j <i$, then define the walk after the next step to be $(x_0, \ldots, x_j)$ (we erased the loop $(x_j,x_{j+1}, x_{j+2}, \ldots, x_i,x_{i+1}=x_j)$).
\end{itemize}
To finish Wilson's algorithm, simply run loop erased random walk from the vertices in any arbitrary order, and then after each step, assume that the  loop erased paths sampled so far along with $\partial$ is the new updated boundary.

In an infinite transient graph, the wired uniform spanning forests can be sampled using a very similar method called Wilson's algorithm \emph{rooted at infinity} (see \cite{LyonsPeres}). The algorithm is the same except we run the loop erased random walk until it hits one of the previously sampled branch or is run forever. Since the walk is transient, this process is well defined.

\subsection{Embeddings and circle packing} \label{sec:embedding}


Assume  that $\Gamma$ is transient, one ended triangulation with degrees uniformly bounded above by $\Delta$. 
\begin{figure}
\centering
\includegraphics[scale = 0.3]{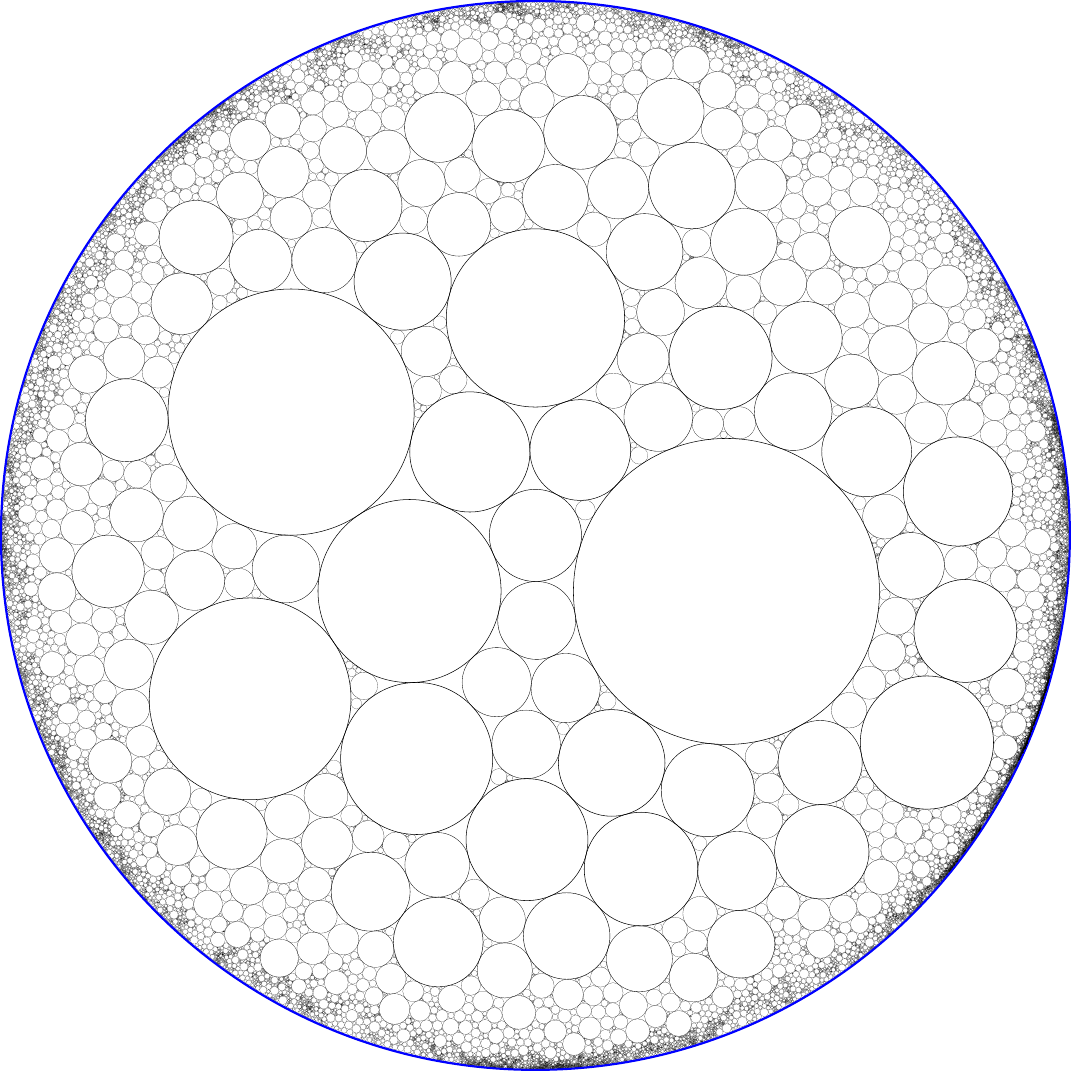}
\caption{Circle packing of $(\Gamma, \rho)$. \copyright{T. Hutchcroft.}}
\end{figure}

A \textbf{circle packing} is a way to draw $\Gamma$ in the plane so that each vertex is the center of a circle, and two adjacent vertices are connected by an edge if and only if the circles touch each other. The carrier of a circle packing is the union of all the discs in the packing
together with the curved triangular regions enclosed between each triplet of circles
corresponding to a face (the interstices). Given some planar domain $D$, we say
that a circle packing is in $D$ if its carrier is $D$. It is known \cite{HeSc} that the triangulations we are considering can be circle packed with the unit disc as the carrier (such triangulations are VEL hyperbolic).
Furthermore,  the rigidity theorem of Schramm \cite{Schramm91} tells us that the circle packing with carrier  the unit open disc $\D$ is unique up to Mobius transforms of the unit disc $\D$, that is, if we take two circle packings in the unit disc they can be mapped to each other by an appropriate Mobius transform. Call $\sf \Gamma$ the embedded graph. We shall need the following theorem about simple random walk on $\sf \Gamma$.
\begin{thm}[\cite{BS96a,ABGN14,AHNRCP15}]\label{thm:RW_circle_packing}
Let $(X_n)_{n \ge 0}$ be a simple random walk on $\Gamma$. Let $z(X_n)$ denote the location of the walk in the embedding $\sf \Gamma$. Then the following holds.\begin{enumerate}[a.]
\item The sequence $z(X_n)$ almost surely converges to a point in $\partial \D$. Call this point $z(X_\infty)$.
\item The law of $z(X_\infty)$ is fully supported on $\partial \D$.
\item The law of $z(X_\infty)$ is nonatomic.
\end{enumerate}
\end{thm}

We denote by $B(x,r)$ the disc of radius $r$ centred at $x$ in the plane.
A consequence of \Cref{thm:RW_circle_packing} is the following lemma:
\begin{lemma}[Beurling]\label{lem:beurling}
 For all $\eps, \eps'>0$ there exists a $\delta>0$ such that with probability at least $1-\eps$, a simple random walk on $\sf \Gamma$ started at a vertex outside $(1-\delta)\D$  never exits the Euclidean ball of radius $\eps'$ from it's starting point.
 \end{lemma}
\begin{proof}
By applying a Mobius transform which maps the starting vertex to 0, an equivalent statement is the following. Fix $x \in \partial \D$. For all $\eps>0$, there exists a $\delta'$ such that simple random walk started at 0 never enters $B(x, \delta')$ with probability at least $1-\eps$. This is a simple consequence of nonatomicity of the exit measure (\Cref{thm:RW_circle_packing}, item c.) since if we assume the negation of the statement, we can conclude that the exit measure has an atom at $x$.
\end{proof}
We need a slight tweak of this lemma for future reference where the point $x$ is selected randomly using an independent random walk. The proof is identical to that of \Cref{lem:beurling}.

\begin{lemma}\label{lem:beurling_rw}
Let $(X_n)_{n \ge 0}$ be a simple random walk started from $\rho$ and let $z(X_\infty) \in \partial \D$ be its limit point.  For all $\eps, \eps'>0$ there exists a $\delta>0$ such that a simple random walk on $\sf \Gamma$ started at a vertex outside $(1-\delta)\D \cup B(z(X_\infty), 2\eps')$   never exits the Euclidean ball of radius $\eps'$ from its starting point with probability at least $1-\eps$.

\end{lemma}

Now suppose we draw the dual graph and create the Temperleyan graph graph $G$ on the embedding $\sf \Gamma$ as in \Cref{sec:height}, and call the embedded graph $\sf G$. We do so in such a way that the embedding is proper. Recall the limits of the unoriented wired and the free UST from \Cref{thm:UST_limit} which we denoted by $T, T^\dagger$ respectively. This yields an embedding of the triplet $(\sf G, \sf T, \sf T^\dagger)$. Recall that every component of $T^\dagger$ has a unique end almost surely. For the following 

\begin{lemma}\label{lem:convergence_WUSF_ray}
Every end of $\sf T$ converges to a point in $\partial \D$.
\end{lemma}
\begin{proof}
This simply follows from the convergence of random walk proved in \Cref{thm:RW_circle_packing}, item a. and the Wilson's algorithm rooted at infinity.
\end{proof}


\begin{prop}\label{prop:convergence_FUSF_ray}
  For every $x \in \partial \D$,
there is a unique end of $\sf T^\dagger$ which converges to $x$ almost surely.  Furthermore, every end of $\sf T^\dagger$ converges to a unique point in $\partial \D$ almost surely. 
\end{prop}
\begin{proof}
It is enough to prove the result for $x=1$.
We first prove existence.
We define a sequence of parameters $(\theta_k, \eps_k, \delta_k)_{k \ge 1}$ satisfying the property that a simple random walk started at $B(e^{i\theta_k}, \delta_k) \cap \D$ does not leave $B(e^{i\theta_k}, \eps_k)$ with probability at least $1-2^{-k}$, $\theta_k>0$ and $e^{i\theta_k} \to 1$ and $B(e^{i\theta_k}, \eps_k)$ are disjoint for different $k$s. The existence of such a sequence is easily seen to be guaranteed by the Beurling estimate of \Cref{thm:RW_circle_packing}.

We now perform Wilson's algorithm from a single arbitrary vertex inside $B(e^{i\theta_k}, \delta_k)$ and $B(e^{-i\theta_k}, \delta_k)$ for each $k$. By the Borel Cantelli lemma, almost surely, for all but finitely many $k$, the simple random walk  does not leave $B(e^{i\theta_k}, \eps_k)$. Now observe that the loop erasures of these walks must be disjoint and consequently, all the loop erasures must belong to different components of the WUSF. Consequently, there must be one end of the dual FUSF which separates the WUSF components corresponding to $k$,  for all large enough $k$ and this guarantees the existence of at least one branch of the FUSF which converges to $1$ almost surely.

Now we prove uniqueness. Let $\cB$ be the event that two distinct branches of the FUSF converges to 1. If $\P(\cB)>0$, there must be an $r\in (0,1)$ such that two distinct rays in the FUSF start inside $B(0,r)$ and converge to $1$ with positive probability. Now fix $0<\delta<1-r$, and perform Wilson's algorithm from all the vertices on the boundary of ${\sf \Gamma} \cap B(0,1-\delta)$. Since the exit measure for simple random walk is nonatomic by \Cref{thm:RW_circle_packing}, none of the branches created by such walks converge to $1$ almost surely. Hence for the branches of the FUSF which separate before $B(0,\delta)$, only one of them can converge to $1$ almost surely. This is a contradiction to the fact that $\cP(\cB)>0$, hence $\cP(\cB) = 0$.

Now we prove that every end of $\sf T^\dagger$ converges to a point in $\partial D$. Suppose not, that is, assume that with positive probability there is a ray of the FUSF which does not converge to $\partial \D$. For every vertex in a ray $(v_0,v_1,\ldots,)$ we can find a unique straight line joining 0 and $\partial \D$ going through $v_i$ and ending in $\theta_i \in \partial \D$. If a ray does not converge, then neither does $\theta_i$ (in the usual topology). Thus with positive probability, there is an $\eps>0$ and two disjoint arcs $I_1,I_2 \subset \partial \D$ of length at least $\eps$ separated by an arc of length at least $\eps$, and a ray of $\sf T^\dagger$ with $\theta_i$s such that infinitely many of them are in both components of $\partial \D \setminus \{I_1 \cup I_2\}$. Let $C_i$ be the cone obtained by connecting the endpoints of $I_i$ with $0$. If the above event happens for a ray $\gamma$ of $\sf T^\dagger$, then for all $\delta>0$, $\gamma$ crosses one of $A_i(\delta):= C_i \cap (\D \setminus (1-\delta )\D)$. However using Beurling's estimate \Cref{lem:beurling}, we can find a sequence of vertices $(x_k)_{k \ge 0}$ and $\delta_k\to 0$ with $x_k \in A_i(\delta_k)$ such that almost surely for all but finitely many $k$, the branch if $\sf T$ sampled from $x_k$ crosses $A_i(\delta_k)$ from its inner to outer boundary. By duality, both events cannot happen, and we have a contradiction.
\end{proof}

\begin{prop}\label{prop:convergence_FUSF_ray_random}
  Let $z(X_\infty)$ be as in \Cref{lem:beurling_rw} and let $\sf T$ be independent of $X$.
There is a unique end of $\sf T$ which converges to $z(X_\infty)$ almost surely.  
\end{prop}
\begin{proof}
The proof is exactly the same as that of \Cref{prop:convergence_FUSF_ray} except we use \Cref{lem:beurling_rw} in place of \Cref{lem:beurling}.
\end{proof}
We finish with the following remark which was the outcome of a discussion with Tom Hutchcroft.
\begin{remark}\label{rmk:end}
One may wonder how the ends of the WUSF and the FUSF are related with $\partial \D$. Clearly the ends of WUSF correspond to countably many points $\Xi$. With some work, we believe it can be shown that the natural map from the ends of the FUSF to $\partial \D$ is a.s. a continuous map which is onto and is 1-1 except at $\Xi$ which corresponds to two ends of the FUSF and one end of the WUSF. We do not pursue this further as we don't need it in what follows.
\end{remark}

%
%

\section{Height function in infinite volume}
\subsection{Winding and height function in infinite volume}\label{sec:infinite_limit}

In this section we prove \Cref{thm:main0} and the finiteness of the height function part of \Cref{thm:height}.
Assume  that $\Gamma$ is transient, one ended triangulation with degrees uniformly bounded above by $\Delta$ and we circle pack it in an arbitrary way.
Using \Cref{prop:convergence_FUSF_ray}, we can conclude that the weak limit of the uniform dimers exist in the following sense. Let $\sf \Gamma^\dagger_r$ be the subgraph of $\sf \Gamma^\dagger$ induced by  ${\sf \Gamma}^\dagger \cap B(0,r) $. Now let $\mathfrak b^1_r$ be the vertex of $\partial \sf G_r$ chosen so that $\mathfrak b^1_r \to 1$ as $r \to \infty$ (e.g. choose the vertex that is closest to the segment $(0,1]$ in $\sf G_r$). Let $\sf G^1_r$ be the graph obtained in this way (see \Cref{fig:temp}). Let $\mu_r^1$ denote the uniform measure on dimer covers of $\sf G_r^1$.

Let $\sf T, \sf T^\dagger$ be the infinite volume WUSF and FUSF embedded in the disc using the circle packing, the existence of  which is guaranteed by \Cref{thm:UST_limit} and furthermore the limit does not depend on the exhaustion chosen.
Let $\vec {\sf T}^{\dagger,1}$ denote the oriented tree obtained by orienting $\sf T^\dagger$ towards the unique end converging to $1$ (as guaranteed by \Cref{prop:convergence_FUSF_ray}).
Let $\mu^1$ be the probability measure obtained on dimer covers of $\sf G$ by applying Temperleyan bijection to $\vec {\sf T}^{\dagger,1}, \vec {\sf T}$.

\begin{proof}[Proof of \Cref{thm:main0}]
Clearly using a rotation, we can assume $x = 1$.
Using the Temperleyan bijection, the dimer covering of ${\sf G}^1_r$ is in bijection with the spanning trees of ${\sf \Gamma}_r$ and ${\sf \Gamma}_r^\dagger$, with the orientation of the spanning tree of $\sf \Gamma_r^\dagger$ towards a branch converging to $\mathfrak b_r$. Since we know the weak limits of the unoriented spanning trees do exist, any subsequential limit will correspond to the unoriented FUSF and WUSF pair of $\sf \Gamma^\dagger,\sf \Gamma$ respectively, where the FUSF is oriented towards an end converging to 1. However, since there is a unique such branch as proved in \Cref{prop:convergence_FUSF_ray}, all subsequential limits must coincide.
\end{proof}

To understand the convergence of  height functions, we need to cast the convergence result of \Cref{thm:main0} into the setup of the extended graph as in \Cref{sec:height}. Consider the portion of the circle packing of ${\sf \Gamma}_r$. We draw the boundary cycle $C_\partial$ inside the unit disc. This ensures that the Euclidean distance between $C_\partial $ and ${\sf \Gamma}_r$ converges to 0. Consequently, using \Cref{lem:convergence_WUSF_ray}, $W_{\i} ({\gamma_{\sf T_r}} (f))$ is almost surely continuous in $r$ for any face $f$ in $\sf G_r$. This allows us to conclude the following. In $\Gamma$, let  $\gamma_{\sf T} (f)$ denote the curve obtained by concetenating the part of the diagonal of $f$ connecting $m(f)$ and the vertex of $\Gamma$ incident to $f$, the infinite branch $\gamma$ of $\sf T$ (which we know converges to a point $x_f \in \partial \D$ via \Cref{lem:convergence_WUSF_ray}), and the arc $(x_f, 1)$ going in anti-clockwise direction from $x_f$ to 1. Since every branch of $\sf T$ converges, it is clear that $|W_{\i}(\gamma_{\sf T} (f))|<\infty$ for every $f$ almost surely. In other words, it immediately follows from \Cref{lem:height_winding_finite} that the height function is \emph{localized}. We obtain the following formula for height function using the topological winding, which will be more useful, as follows
\begin{equation}
h(f) = W_{\i}(\gamma_{T_{\sf m}}(f)) = W(\gamma_{T_{\sf m}}(f), m(f)) + W(\gamma_{T_{\sf m}}(f), 1)\label{eq:winding_topoloical}
\end{equation}
The second term above is simply the angle made by a straight line joining $(m(f),1)$ and $(m(f), x_f)$. Because of this reason, in the estimates that follow (particularly in \Cref{sec:height_percolation}) we often ignore the second term.

 It is also clear that the point $1$ on the boundary an be replaced by any $x \in \partial \D$ with no change in the definitions and the convergence statements. We denote by $(h^x(f))_{f \in F(\sf G_r)}$ denote the height function corresponding to $\mu_r^x$.
 We record this in the next theorem, which proves the first part of \ref{thm:height} as well.

\begin{thm}\label{thm:ht_func_conv}
Let $(h^x(f))_{f \in F({\sf G}^x_r)}$ denote the height function corresponding to $\mu^x_r$. Then $(h^x(f))_{f \in F({\sf G}^x_r)} $ converges to $(h^x(f))_{f \in F(\sf G)}$ in law and the limit is independent of the choice of the exhaustion used. Furthermore, $|h^x(f)| <\infty$ for all $f \in F(\sf G)$ almost surely, i.e., the height function is localized.
\end{thm}

We remark that the height function obtained  in \Cref{thm:ht_func_conv} is entirely dependent on the particular choice of the embedding used. However the theory developed in \cite{BLR16}, as described in \Cref{sec:winding}, gives us an explicit formula for the change in height function under M\"obius maps (in fact the theory works for any conformal map, but we stick to M\"obius maps in our context). 
As a corollary of \Cref{lem:winding_change_conformal}, for example, we immediately obtain that for any two different embeddings  of $\Gamma$ related by a M\"obius map $\varphi$ which maps $x$ to $x$,  
\begin{equation}
(h^x(\varphi(f)))_{\varphi(f) \in F(\sf \varphi(\sf G))}  = (h^x(f)  - \arg_{\varphi'(\D)}(\gamma'(m(f))))_{f \in F(\sf G)}.\label{eq:height_mobius}
\end{equation}
We quickly remark that $\arg_{\varphi'(\D)}$ makes sense as $\varphi'(\D)$ is a domain not containing $0$ since $\varphi$ is conformal.
Indeed, since the endpoint point of every branch $\gamma_{\sf T}$ of $\sf T$ as described  is $x$ in both embeddings,   the term corresponding to $\arg_{\varphi'(\D)}(\gamma'(1))$ in \Cref{lem:winding_change_conformal} is always 0. Another simple observation one can make directly from the above formula is that the height fluctuations
$$
(h^x(f) - \E(h^x(f)))_{f \in  F(\sf G)}
$$
is independent of the embedding, as it should be. Of more relevance to us would be the observation that the difference between the heights of two independent copies of dimer cover sampled from $\mu^x$ is independent of the embedding, which can actually be seen directly from the definition.

\subsection{Height function percolation}\label{sec:height_percolation}
Throughout this section, assume  that $\Gamma$ is nonamenable, one ended triangulation with degrees uniformly bounded above by $\Delta$ and let $G$ be as in \Cref{thm:ht_func_conv}.
Recall that a graph is \textbf{nonamenable} if 
\begin{equation}
\inf_{S\subset V} \frac{|\partial_E S|}{|S|_E} \ge {\sf h} >0\label{eq:nonamenable}
\end{equation}
where the infimum above is over all finite subsets of $V$, $\partial_E S$ is the set of edges connecting a vertex in $S$ to a vertex outside $S$, and $|S|_E$ is the sum of the degrees of the vertices in $S$. The quantity ${\sf h}$ is called the \textbf{Cheeger constant} of $\Gamma$.

In this section we prove the second part of \Cref{thm:height} and \Cref{thm:height_perc}. To recall, we investigate the following question: does there exist $k$ such that $$\cH_k:= \{f\in F(G):|h^1(f)| > k\}$$ has finite components almost surely? Here a component is a subgraph induced by the vertices in the dual of $G$. We prove that the answer is yes.

 It is well known that if ${\sf h}>0$ then the spectral radius $\rho$ of the graph is strictly less than $1$. This is a consequence of the following inequality (\cite{LyonsPeres})
\begin{equation}
\| P \| \le 1-\frac{{\sf h}^2}{2}\label{eq:spectral}
\end{equation}
where $P$ is the Markov transition kernel for the simple random walk on $\Gamma$. 
As an immediate consequence, we obtain the following exponential bound:

\begin{proposition}\label{prop:CV}
 Let $(X_n)_{n \ge 1}$ be a simple random walk on a graph with Cheeger constant $\sf h$ and spectral radius $\rho$. Then for any two vertices $x,y$ in $G$ and $n \ge 1$, 
\begin{equation*}
\P_x(X_n = y) \le \rho^n \le \left(1-\frac{{\sf h}^2}{2}\right)^n
\end{equation*}
\end{proposition}


We need several preparatory lemmas. We circle pack $\Gamma$, chosen arbitrarily, and let $r(v)$ denote the radius of the circle of a vertex $v$. Recall that $\rho<1$ is the spectral radius of $\Gamma$. The following lemma is essentially extracted from an argument in \cite{AHNRCP15}.

\begin{lemma}\label{lem:circle_small}
Let $X_n$ be a simple random walk started from a vertex $v$. Then for all $c>0$
\begin{equation*}
\P(r(X_n) \ge e^{-cn} \text{ for some $n \ge N$}) \le \sum_{n \ge N} (\rho e^{2c})^n.
\end{equation*}
In particular, if $c>0$ is chosen small enough depending on $\rho$, 
\begin{equation*}
\P(r(X_n) \ge e^{-cn} \text{ for some $n \ge N$}) \le \frac{(e^{2c}\rho)^N}{1-e^{2c}\rho}
\end{equation*}

\end{lemma}
\begin{proof}
Considering the total area, the number of circles with radius at least $e^{-cn}$ is at most $e^{2cn}$. Thus using \Cref{prop:CV}, the probability that $X_n$ is in a circle of radius at least $e^{-cn}$ is at most $e^{2cn}\rho^n$. The proof is easily completed using union bound.
\end{proof}

\begin{lemma}\label{lem:path}
There exists a constant $C>1$ such that the following is true for all vertices $v $ in $\sf \Gamma$. There exists an infinite path $P$ in $\sf \Gamma$ started at $v$ such that
\begin{itemize}
\item  for all $r \ge 1$, the length of the path between the first exit of $B_{\sf \Gamma}(\rho,r)$ and last exit of $B_{\sf \Gamma}(\rho,r')$ is at most $C(r'-r)$ for all $r'>r$,
\item $\sup_{Q \subset P} |W(Q, v)| \le C$, where the supremum is taken over any continuous segment $Q$ of $P$ and $W(Q, v)$ is the topological winding of $Q$ seen from $v$.
\end{itemize}
There also exists an infinite path in the dual graph $\Gamma^\dagger$ with the above properties.
\end{lemma}
\begin{proof}
We first prove that there exists a $C_0>0$ such that for all $v$, with probability at least $3/4$, the loop erasure of the infinite simple random walk started at $v$ has has winding at most $C_0$. Run a simple random walk $(X_n)_{n \ge 0} $ started from $v$, and let $r(X_n)$ denote the radius of the circle in the circle packing with centre $X_n$. Choosing $c$ small so that $e^{2c}\rho<1$, and employing \Cref{lem:circle_small}, we conclude that the diameter of the random walk trajectory $(X_n)_{n \ge N_0}$ is at most $1/10$ with probability at least $3/4$. Now consider the loop erasure of the walk. Let $(Y_i)_{0 \le i \le p}$ be the loop erasure of $(X_i)_{0 \le i \le N_0}$. Let $j$ be the last index of $(Y_i)_{0 \le i \le p}$ which is hit by $(X_j)_{j \ge N_0}$. Clearly, the winding of any continuous segment of $(Y_i)_{0 \le i \le j}$ is at most
$2\pi j \le 2 \pi N_0$ in absolute value. Any continuous segment from the remaining portion of the loop erasure is a subset of $(X_j)_{j\ge N_0}$ which has diameter at most $1/10$ and hence cannot wind more than $2\pi$ around $v$. Overall the winding of any continuous segment of the loop erasure of the walk is at most $2\pi(1+N_0)$. We now can take $C_0 = 2\pi(1 +N_0)$.

We now concentrate on the first item.
Let $(X_n)_{n \ge 0}$ denote the simple random walk trajectory started from $\rho$. 
By an union bound, the probability that $X_n \in B_\Gamma(v, \lfloor \eps n \rfloor)$ for some  $n \ge N_0(\rho,\Delta)$ is at most  
$$
\sum_{n \ge N_0} (\Delta)^{\eps n}\rho^n <\frac12
$$
 for a small enough choice of $\eps = \eps(\rho, \Delta)$ satisfying $\Delta^\eps \rho <1$. Union bounding this estimate with that in the previous paragraph, with probability at least $1/4$, $X_n \not \in B_\Gamma(v, \lfloor \eps n\rfloor)$ for all $n \ge N_0$ and the loop erasure satisfies the second item in the lemma with the chosen constant $C_0$.

 Now join $v$ with a vertex $Y_j$ in the path $Y$ which is closest to $v$ using a geodesic and then concatenate it with $(Y_i)_{i\ge j}$. Call this new path $Y'$. 
  On the event (with probability at least $1/4$ described above), we claim that $Y'$ satisfies all the conditions of the first item of the lemma with $C = \max (\Delta^{N_0}/N_0, \frac1\eps, C_0)$.
To see this, note that for any $N_0\le r<r'$ the portion of $Y'$ between first exit of $B_\Gamma(v,r)$ and the last exit of $B_\Gamma(v,r')$ is a subset of $(X_k: r \le k \le r'/\eps)$. On the other hand if $r<N_0$ and $r<r'$, the last exit of $B_\Gamma(r')$ may occurs after or before the first exit of $B_\Gamma(N_0)$. In the former, we can decompose the path into that between first exit of $B_\Gamma(r')$ and the first exit of $B_\Gamma(N_0)$ and then the rest. The first part has trivially a length at most the volume of $B_\Gamma(N_0)$ which is at most $\Delta^{N_0}$. Similarly in the latter case, we can also bound the whole length by $\Delta^{N_0}$. The result now follows by combining this with the upper bound obtained in the $N_0\le r<r'$ case. 

For the dual graph, using the faces adjacent to the edges of $P$, it is easy to see that we can find a path $P^\dagger$ with $C$ in the second item replaced by $C\Delta$ and $C$ in the first item is replaced by $C+2\pi$.
\end{proof}

We next prove that in a nonamenable graph the loop erasure of any path of length at most $C^r$ can wind at most linearly in $r$  in a strong sense (this is a special property of nonamenable graphs). The path in  the dual in \Cref{lem:path} is somewhat convenient to use as a `reference cut' to compute the winding of a path in the primal.
\begin{figure}
\centering
\includegraphics[scale = 0.7]{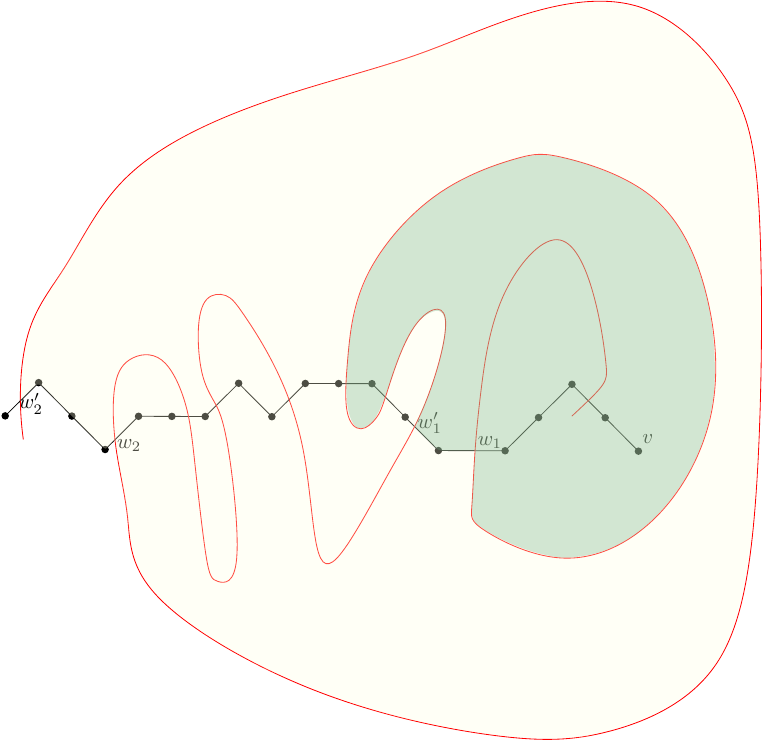}
\caption{An illustration of the proof of \Cref{lem:winding_lemma}. Two consecutive sign pairs are shown. The path $P^\dagger$ is in black and $\eta$ is in red. The domain $D_1$ is shaded light blue and $D_2$ is shaded light yellow.}
\end{figure}

\begin{lemma}\label{lem:winding_lemma}
There exists a constant $C$ such that the following holds.
Let $(v=v_0,v_1, \ldots, v_{k})$ be any path. Let $\gamma$ denote the loop erasure of thus path. Then 
$$
\sup_{\eta \subset \gamma} |W(\eta,v) | \le C\log k
$$
where the supremum is taken over all continuous segments $\eta$ of $\gamma$.
\end{lemma}
\begin{proof}
Take a continuous segment $\eta \subset \gamma$.
Take the path $P^\dagger$ as described just before the lemma started from a face adjacent to $v$, and connect the starting point to $v$ to obtain a path started from $v$. We keep the name $P^\dagger$ for this new path admitting an abuse of notation. 
Orient $P^\dagger$ from $v$ to $\infty$. Let $C_0$ be as in \Cref{lem:path} for $P^\dagger$.
Assign a sign $+$ to an edge of $P^\dagger$ if $\eta$ crosses it from right to left, and $-$ sign if $\eta$ crosses it from left to right. Otherwise an edge of $P^\dagger$ receives no sign. This gives a sequence of signed edges along $P^\dagger$.  We say two edges $e,e'$ is a  consecutive sign pair if there is no other edge $e''$ between $e,e'$ in $P^\dagger$ which has a sign, and $e,e'$ both have the same sign. 

Order the signed edges using the orientation of $P^\dagger$, and consider the first signed pair and assume (without loss of generality) that they are both $+$. Let $e^\dagger, (e')^\dagger$ be the oriented duals of $e,e'$ crossing $P^\dagger$ from right to left and  let $w,w'$ be the points of $e,e'$ at which they cross it. By assumption, both $e^\dagger, (e')^\dagger$ are in $\eta$ and assume $\eta$ hits $e^\dagger$ first and then $(e')^\dagger$. Then the portion of $\eta$ between $w,w'$ (call it $\eta(w,w')$) and the portion of $P^\dagger$ between $w,w'$ form a simple closed loop $L$ enclosing a domain $D$. Furthermore, the head of $(e')^\dagger$ must be outside $D$ as by construction the only way out of $D$ is through $(w,w')$ which is not allowed. Consequently, the tail of $e^\dagger$ is in the domain, and hence so is $v$ (as $v$ is joined to the tail of $e^\dagger$ by $\gamma$). Thus the winding contributed by this portion of $\eta$ is at most $2\pi + C_0$ since the winding of the segment $w,w'$ is at most $C_0$ by construction. Now we can continuously deform $\eta$ between $w$ and $w'$ so that it crosses $P^\dagger$ only at $w$ and $w'$, keeping the winding constant. In effect, we can remove the signs from the edges other than $e,e'$ crossed by $\eta (w,w')$.

We now iterate this procedure. Consider the next consecutive signed pair (after the deletion of the signs in the first step). This gives us another simple loop $L'$ enclosing a domain $D'$ containing $v$ which only crosses $P^\dagger$ at two edges, and it is easy to see that $D'$ contains $D$. Furthermore, the portion of $\eta$ between the end point of $L$ and the start point of $L'$ (after the deformation of $L'$) crosses $P^\dagger$ with alternating signs and thus contributes at most $2\pi$ to the winding. Indeed, to see this, call this portion $\eta_{L,L'}$. Suppose  $\eta_{L,L'}$ creates a consecutive sign pair, which creates a domain $D''$ lying between $D'$ and $D$. By constriction, the portion of $L'$ in $\eta$, before the deformation, can enter and exit $D''$ only an even number of times, which cannot maintain the alternating signs of the crossings required to ensure that $L'$ is created from the first consecutive signed pair.

 Thus for each iteration, we get a nested sequence of loops $L_1,L_2,\ldots, L_q$ and the total winding of $\eta$ from its start point to the end of $L_q$ is at most $(C_0+2\pi)q+2\pi$ and each $L_i$ enclose a domain $D_i$ containing $v$. Now note that any path started at $v$ exiting $D_q$ must have length at least $q$ simply because it has to cross all the loops $L_i$. Thus $B_\Gamma (v,q/2) \subset D_q$. Thus the loop $L_{q+1}$ for the next signed pair encloses a domain $D_{q+1}$ of volume at least $|B_\Gamma (v,q/2)| \ge (1+\alpha({\sf h}))^q$ for some $\alpha({\sf h})>0$ depending on the positive Cheeger constant. Let $\eta_{q+1}$ be the portion of $\eta$ in $L_{q+1}$ and let $(w_{q+1}, w'_{q+1})$ be the portion of $P^\dagger$ in $L_{q+1}$. The length of $(w_{q+1}, w'_{q+1})$ is at most $C_0(d_\Gamma(v, w'_{q+1}) - d_\Gamma(v, w_{q+1}))$ (by the second property of $P^\dagger$ as in \Cref{lem:path}) and the length of $\eta_{q+1}$ is at least $(d_\Gamma(v, w'_{q+1}) - d_\Gamma(v, w_{q+1}))$ \footnote{Although $w_q,w_{q+1}$ are midpoints of edges in $\Gamma$, we can use the distance to one of the endpoints of the edges. We keep this slightly erroneous notation for clarity.}. Furthermore, by nonamenability, the length of $L_{q+1}$ must be at least ${\alpha{(\sf h)}} |B_\Gamma (v,q/2)|$. From here we can easily conclude that the length of $\eta_q$ is at least $\alpha({\sf h})\frac{(1+\alpha({\sf h}))^q}{1+C_0}.$ Since this quantity is at most $k$, we conclude that $q \le C\log k$ for some appropriate constant $C$. Since we already concluded that $|W(\eta, v)|$ is at most $(C_0+2\pi)q+2\pi$, the result follows.
\end{proof}
We are now ready to prove the core ingredient of the second item of \Cref{thm:height}.
\begin{lemma}\label{lem:height_tail}
There exists a $c,c'>0$ such that the following holds.
Let $\gamma_v$ denote the branch of the WUSF started at $v$.
Then for all $r \ge 1$
$$
\P(\sup_{\eta \subset \gamma_v} |W(\eta, \rho)|>r) \le \exp(-ce^{c'r}).
$$
where the supremum is over all connected segments of $\gamma_v$.
\end{lemma}
\begin{proof}
We use a similar estimate as in \Cref{lem:path} and use the notations there. Simply note that using \Cref{lem:circle_small}, with probability at least $1-\exp(-cC^r)$, none of the radii $r(X_n)$ has length more than $e^{-cn}$ for all $n \ge C^r$. Thus on this event the diameter of the trajectory $\{r(X_n)\}_{n \ge C^r}$ is at most $1/10$ for all large enough choice of $C$ whose loop erasure do not contribute more than $2\pi$ to winding. Furthemore, the winding of any continuous segment of the loop erasure of the first $C^r$ steps is (deterministically) at most $C'r$ using \Cref{lem:winding_lemma}. Combining these two estimates, the result follows.
\end{proof}
\begin{proof}[Proof of \Cref{thm:height}]
The first part is proved in \Cref{thm:ht_func_conv}. The estimate of the tail follows from \Cref{lem:height_tail} and \eqref{eq:winding_topoloical}.
\end{proof}

We now record an immediate corollary of \Cref{thm:height}.
\begin{corollary}\label{cor:height_moment}
For all $j \ge 1$ there exists an $M(j)$ such that for any face $f \in F(G)$,
$$
\E(|h^1(f)|^j) \le M(j).
$$
\end{corollary}
\bigskip
\bigskip
We are now ready to prove \Cref{thm:height_perc}.
\begin{proof}[Proof of \Cref{thm:height_perc}]
The broad idea is to run Wilson's algorithm and then dominate $\cH_k$ by a subcritical branching process using the tail estimates. However, we need to run the algorithm in a subtle way to control the size of $\cH_k$.

From every vertex $v$ run independent copies simple random walks $(X^v_n)_{n \ge 0}$. Let $N_v = \inf \{n: \sum_{k \ge n}r(X_k)<1/10\}$, where $r(X_k) $ is, as before, the radius of the circles corresponding to  the walk. We already know $N_v$ has exponential tail using \Cref{lem:circle_small}. Let $S^v_k = \{X_i^v: 0 \le i \le k\}$. Let $S_0 = S^\rho_{N_\rho}$ if $N_\rho \ge k$ and let $A_1$ be the set of vertices not in $S_0$ but having a neighbour in $S_0$. Otherwise $S_i  = \emptyset $ for all $i \ge 0$ and $A_i = \emptyset$ for all $i \ge 1$.  Let $S_1 = \cup_{v \in A_1, N_v \ge k}S^v_{N_v}$. If $S_1  = \emptyset $ declare $A_i = \emptyset$ for all $i \ge 2$. Otherwise, iterate the process: inductively, if $A_j \neq \emptyset$, declare $S_j:=\cup_{v \in A_j, N_v \ge k}S^v_{N_v}$ and if $S_j = \emptyset$, declare $\cup_{i \ge j}S_i = \emptyset$ and $\cup_{i \ge j+1} A_i = \emptyset$. Otherwise, define $A_{j+1} $ to be the set of vertices not in $S_j $ but has a neighbour in $S_j$.

Observe that if $S_j = \emptyset $ for some $j$, we have found a cutset $\cC$ separating $\rho$ from $\infty$, where $N_v\le k$ for all $v \in \cC$. Now we use the fact that Wilson's algorithm rooted at $\infty$ is a deterministic function $\Phi(X^v)$ for $v \in \Gamma$, and the order in which the loop erasure is performed will always create the same output of $\Phi$. This is a simple consequence of the cycle popping algorithm due to Wilson, see \cite[Proposition 10.1]{LyonsPeres} for a version in finite graphs, and \cite[Theorem 4.1]{ARS_loop} for a version in infinite graphs. If we loop erase the random walk from any vertex of $\cC$, using \Cref{lem:winding_lemma} and the definition of $N_v$, the winding of any continuous segment of the spanning tree branches is deterministically at most $C\log k$. Thus on the event $S_j = \emptyset $ for some $j$, the cluster of $\cH_{C\log k+1}$ containing $\rho$ is finite.

It remains to show that $S_j = \emptyset $ for some $j$ holds almost surely for large enough choice of $k$. To that end, we dominate $(A_j)_{j \ge 0} $ by a subcritical branching process. Indeed if $A_j \neq \emptyset$, each vertex $v$ in $A_j$ creates at most $\Delta N_v1_{N_v \ge k}$ many vertices in $A_{j+1}$ independently of each other. Now note that using \Cref{lem:circle_small}, $$\E(N_v1_{N_v \ge k} ) = \sum_{t \ge k} Ct e^{-ct} <1/2$$
for a large enough choice of $k$.
\end{proof}
\begin{remark}
From the domination by the subcritical branching process, it is possible to deduce exponential tail of the volume of $\cH_k$ for large enough $k$.
\end{remark}

A very similar proof yields the following.
\begin{corollary}\label{cor:ind_height_perc}
Take two independent dimer coverings and let $\cH_k, \cH_k'$ be their corresponding height function percolation components of height at least $k$. Then there exists a $k_0>0$ such that for all $k \ge k_0$ $\cH_k \cup \cH'_k$ has no infinite component almost surely.
\end{corollary}
\begin{proof}
We need to modify the definition of $N_v$ in the proof of \Cref{thm:height_perc} to 
$$
N_v = \inf \{n: \sum_{k \ge n}r(X_k)<1/10, \sum_{k \ge n}r(X'_k)<1/10\}
$$
for two independent copies of simple random walk $X^v,(X^v)'$. 
\end{proof}


We now estimate a certain correlation  between the height functions which will be key in proving the nonexistence of bi-infinite paths in double dimers.  Let $z(X_\infty)$ be the limit point of a random walk started at $\rho$. Let $\Delta h^{z(X_\infty)}$ be the difference in height functions obtained from two i.i.d.\ copies of dimer covers sampled from $\mu^{z(X_\infty)}$. Furthermore assume that conditioned on the limit point $z(X_\infty)$, the dimers are sampled independently of the walk $X$.

\begin{lemma}\label{lem:correlation}
For all $\eps, \eta>0,  k \in \N$ the following is true. There exists a $\delta>0$ such that for all $f_1,\ldots, f_k \in F(\sf G)$ not in $(1-\delta)\D$ and $|f_i - f_j|>\eta$ for all $i \neq j$, 
$$
|\E(\Delta h^{z(X_\infty)}(f_1) \Delta h^{z(X_\infty)}(f_2)\ldots \Delta h^{z(X_\infty)}(f_k)) |<\eps.
$$
\end{lemma}
\begin{proof}
Suppose the height functions are sampled from two independent copies of the WUSF $\sf T, \sf T'$.
Condition on the branches $\gamma_i, \gamma_i'$ of $\sf T$ started from $f_1,\ldots, f_{k-1}$. Using \Cref{lem:beurling} we conclude the following. For all $\eps'$ and $\eta' \ll \eta$, we can choose a $\delta$ such that  with probability at least $1-\eps'$ all the branches $\gamma_i, \gamma_i'$ do not leave $B(f_i, \eta')$ for all $1 \le i \le k-1$. Let $\cG$ be this event. Using \Cref{cor:height_moment} we can conclude that 
\begin{multline}
|\E(\Delta h^{z(X_\infty)}(f_1) \Delta h^{z(X_\infty)}(f_2)\ldots \Delta h^{z(X_\infty)}(f_{k-1}))  \\- \E(\Delta h^{z(X_\infty)}(f_1) \Delta h^{z(X_\infty)}(f_2)\ldots \Delta h^{z(X_\infty)}(f_{k-1})) 1_{\cG}| \le C^k \eps'\label{eq:correlation1}
\end{multline}
for some $C>0$ independent of everything else. Now let $\cF$ be the sigma algebra generated by $\gamma_1,\ldots, \gamma_{k-1},\gamma'_1,\ldots, \gamma'_{k-1}$. Conditioned on $\cF$ such that $\cG$ holds, we can next sample the branches $\gamma_k,\gamma_k'$. Using Beurling, for a small enough choice of $\eta'$, we can ensure that the simple random walk started at $f_k$ does not enter any of the neighbourhoods $B(f_j, \eta'')$ , $1 \le j \le k-1$ with probability at least $1-\eps'$. Now define a coupling between $\gamma_k, \gamma_k'$ and LERWs $\tilde \gamma, \tilde \gamma'$ started from $f_k$ where the latter is independent of $\cF$. Run a pair of independnent simple random walks from $\cF_k$, and if it does not hit $\cup_{1 \le i \le k-1} \gamma_i$, define $\tilde \gamma  = \gamma_k$ and $\tilde \gamma' = \gamma_k'$ to be their loop erasure. Otherwise, let $\tilde \gamma$ and $\tilde \gamma'$ to be independent of $\gamma_k, \gamma_k'$.  Let $H$ be the event that $\cup_{1 \le i \le k-1} (\gamma_i \cup \gamma_i')$ is not hit and using the above estimate we get $\P(H^c) <\eps'$. Let $\Delta \tilde h(f_k)$ denote height difference obtained by computing the winding of $\tilde \gamma, \tilde \gamma'$. Then 
\begin{align*}
\E(\Delta h(f_k)|\cF_k)1_\cG &= \E(\Delta h(f_k)1_{H}|\cF_k)1_{\cG} + \E(\Delta h(f_k)1_{H^c}|\cF_k)1_{\cG}\\
& = \E(\Delta \tilde h(f_k)1_{H}|\cF_k)1_{\cG} + \E(\Delta h(f_k)1_{H^c}|\cF_k)1_{\cG}\\
& = \E(\Delta \tilde h(f_k)|\cF_k)1_{\cG}-\E(\Delta \tilde h(f_k)1_{H^c}|\cF_k)1_{\cG} + \E(\Delta h(f_k)1_{H^c}|\cF_k)1_{\cG}
\end{align*}
The first term in the last line above is 0 by independence.
Using \Cref{cor:height_moment} and Cauchy--Schwarz, we can conclude
$$
|\E(\Delta \tilde h(f_k)1_{H^c}|\cF_k) | \le C\eps', \qquad |\E(\Delta h(f_k)1_{H^c}|\cF_k)| \le C\eps'.
$$
Combining this with \eqref{eq:correlation1}, we conclude that
$$
|\E(\Delta h^{z(X_\infty)}(f_1) \Delta h^{z(X_\infty)}(f_2)\ldots \Delta h^{z(X_\infty)}(f_k)) | \le C^k\eps'^2.
$$
The result follows by choosing $\eps'$ so that $C^k\eps'^2 < \eps$.
\end{proof}


\section{Double dimers}
\subsection{Stationary embeddings of dimer decorated graphs}\label{sec:stationary_end}
Let $(\Gamma, \rho)$ be a reversible random graph where $\Gamma$ is a planar, transient, bounded degree, one ended triangulation almost surely. It is known \cite[Example 9.6]{AL07} that $(\Gamma, \rho, T, T^\dagger)$ forms a reversible, marked random graph as well where $T^\dagger$ is the wired uniform spanning tree of $\Gamma$ and $T$ is the free uniform spanning tree of $\Gamma^\dagger$, as in \Cref{thm:UST_limit}. Since the dual is inherently related to an embedding, let us now spend a few words on how to interpret this claim. One may replace any planar embedding of $\Gamma$ by an assignment of numbers to the edges incident to a vertex $v$, which represent the counterclockwise order in which they appear around $v$. This allows us to define the dual of $\Gamma$ in an automorphism invariant way, and consequently $(\Gamma, \rho, T, T^\dagger)$ can be represented as a marked random rooted graph.

We now wish to orient the trees $T, T^\dagger$. However, there is no way to do this in a reversible way for the free spanning tree $T^\dagger$ since it leads to a non-reversible marking. This is related to the fact that a regular tree oriented towards one of its ends forms a nonunimodular transitive graph, see \cite{AL07}. Nevertheless, we claim that there is a way to choose an end of $T^\dagger$ so that the graph with the orientations $(\Gamma, \rho, \vec T, \vec {T}^\dagger)$ forms a stationary random marked graph. This construction goes as follows. Circle pack $(\Gamma, \rho)$ and sample an infinite simple random walk trajectory $(X_n)_{n \ge 0}$ started at $X_0 = \rho$. By \Cref{thm:RW_circle_packing}, the embedding of $(X_n)_{n \ge 0}$ converges to a random point $z(X_\infty)$, which corresponds to a unique end $\xi_{X}$ of $T^\dagger$ via \Cref{prop:convergence_FUSF_ray_random}. Orient $T^\dagger$ towards this end, and call this orientation $(\Gamma, \rho, \vec T, \vec {T}^\dagger_X)$. Observe that this choice of end is independent of the choice of the circle packing (once the random walk trajectory is fixed).

\begin{lemma}\label{lem:Gamma_stationary}
$(\Gamma, \rho, \vec T, \vec {T}^\dagger_X)$ is a stationary marked random graph.
\end{lemma}
\begin{proof}
By stationarity, $$(\Gamma, X_1,  T, T^\dagger)  \stackrel{(d)}{=} (\Gamma, \rho, T, T^\dagger).$$ Since $(X_n)_{n \ge 1}$ has the same law as a simple random walk started from $X_1$, we get that 
$$(\Gamma, X_1,  T, T^\dagger, (X_n)_{n \ge 1})  \stackrel{(d)}{=} (\Gamma, \rho, T, T^\dagger, (X_n)_{n \ge 0}).$$
We can therefore conclude
$$
(\Gamma, \rho, \vec T, \vec {T}^\dagger_X) \stackrel{(d)}{=}(\Gamma, X_1, \vec T, \vec {T}^\dagger_X),
$$
thereby completing the proof.
\end{proof}

Conditioned on $(\Gamma, \rho)$, let $(S,  S^\dagger)$ be an independent copy of $T, T^\dagger$, and orient them using the same random walk $(X_n)_{n \ge 0}$ to obtain $(\Gamma, \rho, \vec S, \vec S^\dagger_X, \vec T, \vec T^\dagger_X).$ Then by exactly the same logic as in the proof of \Cref{lem:Gamma_stationary},
\begin{corollary}\label{cor:iid_trees}
$(\Gamma, \rho, \vec S, \vec S^\dagger_X, \vec T, \vec T^\dagger_X)$ is a stationary marked random graph.
\end{corollary}
\subsection{Proof of \Cref{thm:main}}
In this section we prove \Cref{thm:main} and to that end we assume $(\Gamma, \rho)$ is a reversible, random rooted graph which is a nonamenable, one ended triangulation almost surely with degrees uniformly bounded by $\Delta$.

 Let $(\Gamma, \rho, \vec S, \vec S^\dagger_X, \vec T, \vec T^\dagger_X)$ be as in \Cref{cor:iid_trees}.
Using the Temperleyan bijection, we obtain two i.i.d.\ dimer coverings of $G$, call them $M^S_X, M^T_X$. Then $(\Gamma, \rho, M^S_X, M^T_X )$ form a stationary, marked, random rooted graph simply because Temperleyan bijection describes $M^S_X, M^T_X$ as local automorphism equivariant functions of $\vec S_X, \vec S^\dagger, \vec T_X, \vec T^\dagger$. We first prove the following theorem.
\begin{thm}\label{thm:main_random}
Let $(\Gamma, \rho, M^S_X, M^T_X)$ be as above. Then $M^S_X \Delta M^T_X$ do not contain any bi-infinite path almost surely.
\end{thm}


To define the height function of $M^S_X, M^T_X$, we can arbitrarily circle pack $\Gamma$.
Note that the height functions $h_{M^S_X}$ or $h_{M^T_X}$ have the same law as $h^{z(X_\infty)}$ as described in \Cref{thm:ht_func_conv} where $z(X_\infty)$ is the limit point of the random walk $X$ in the packing. Although the individual height  functions do depend on the choice of the embedding (through the reference flow),  $$\Delta h(f) := h_{M^S_X}(f)  -h_{M^T_X}(f) $$ is independent of the embedding and is a measurable function of $(\Gamma, \rho, M^S_X, M^T_X )$. This is a simple consequence of the change in height function formula \eqref{eq:height_mobius}. In what follows, we let ${\mathbf P}$, $\mathbf E$ and $\mathbf{Var}$ denote the conditional law of the dimers given a realization of $(\Gamma, \rho)$ which is a bounded degree, nonamenable, one ended triangulation.

 \begin{lemma}\label{lem:loop_bernoulli}
 Condition on $(\Gamma, \rho)$.
  Let $\cL(f,f')$ denote the collection of finite loops and potential bi-infinite paths separating $f,f'$ in $M^S_X \Delta M^T_X$. Then $$\Delta h(f) - \Delta h(f')=\sum_{L\in \cL(f,f')}  \xi_{L}$$ in law, where $\xi_L$ are i.i.d.\ 2Bernoulli$(1/2)-1$ random variables independent of $\cL(f,f')$. In particular, 
  
 $$
 \mathbf{Var}(\Delta h(f) - \Delta h(f')) = {\bf E}(|\cL(f,f')|).
 $$
 \end{lemma}
\begin{proof}
Let $G$ be the superimposed graph as described in \Cref{sec:height}.
Observe that for two adjacent faces $f,f'$ separated by an edge $e = (b,w)$ with $b$ being a black vertex, $w$ a white vertex and $f$ being to the left of $\vec e$ oriented from black to white,$$\Delta h(f) - \Delta h(f')  =1_{e \in M^T_X} - 1_{e \in M^S_{X}}.$$

Let us now denote by $(M_n^S, M_n^T)$ the finite approximation of $M_X^S, M^T_X$: take an exhaustion $(G_n)_{n \ge 1}$ with the removed black vertex from the boundary $\mathfrak b_n$ converging to $z(X_\infty)$.
Let $\cL_{n,f,f'}$ denote the set of loops which seperate $f,f'$ in $G_n$. Since every such loop must intersect a large enough ball of radius $R$, $\cL_{n,f,f'} \cap B_G(\rho,R)$ must converge in law to $\cL_{f,f'} \cap B_G(\rho, R)$. Take a fixed path $p$ started at $f$ and ending at $f'$. Let $L_1, L_2, \ldots,L_{|\cL_{n,f,f'}|}$ denote the ordered set of loops crossed by $p$. Consider the event $$\cE_{k,\eps_1,\ldots, \eps_k}:= \{|\cL_{n,f,f'}| = k, \xi_{L_i} = \ve_i, 1 \le i \le k\}$$ where $k \in \N$ and $\eps_i \in \{\pm 1\}$ for all $i$ and $\xi_{L_i}$ denote the change in $\Delta h$ when $L_i$ is crossed by $p$ ($\sum_i \xi_{L_i}$ is of course independent of the choice of $p$). Now we claim that for any other choice of $\eps_i' \in \{\pm 1\}$, $1 \le i \le k$, $$\bP(\cE_{k,\eps_1,\ldots, \eps_k}) =\bP (\cE_{k,\eps_1',\ldots, \eps_k'}).$$ Indeed take a double dimer cover in $\cE_{k,\eps_1,\ldots, \eps_k}$ and for every $i$ such that $\eps_i \neq \eps_i'$, swap the dimers of $M_n$ and $M_n'$ along $L_i$. The image of this mapping gives an element of $\cE_{k,\eps_1',\ldots, \eps_k'})$. This describes an involution, hence a bijection. Since every dimer has cover has the same weight, the probabilities of these events are the same. Thus we conclude that conditioned on $|\cL_{n,f,f'}| = k$, the  $\xi_i$s are i.i.d. $+1$ or $-1$ with equal probability. Since $|\cL_{n,f,f'}|$ converge in law to $|\cL(f,f')|$, we are done.
\end{proof}

\begin{proof}[Proof of \Cref{thm:main_random}]

Let $\cD$ be the event that there exists an infinite path in $M^S_X \Delta M^T_X$ and assume $\P(\cD)>0$. Now we divide into two disjoint cases: let $\cD_\infty$ be the event that there are infinitely many infinite paths, and let $\cD_{\text{fin}}$ denote its complement in $\cD$.

\textbf{Case 1: $\P(\cD_\infty)>0$.} Condition on $(\Gamma, \rho) $ such that $\bP(\cD_\infty)>0$. Let $\cH^S_k$ and $\cH^T_k$ denote the faces with height at least $k$ for $M^S_X$ and $M^T_X$ respectively. The idea is to conclude that there is an infinite component of $\cH^S_k \cup \cH^T_k$ with positive probability for arbitrarily large $k$, which will contradict \Cref{cor:ind_height_perc}.

 Observe that $\Delta h$ is constant along all the faces adjacent to a bi-infinite path which are on the same side of it. Let $\cL_{f,f'}$ be as in the proof of \Cref{lem:loop_bernoulli} (the set of loops which separate $f,f'$). For all $k \ge 1$, one can find $R>0$ and  $f,f'$ such that the event $$\cD(f,f',k):=\{|\cL_{f,f'} \cap B_\Gamma(\rho, R) | \ge k \text{ and $f,f'$ are both adjacent to bi-infinite paths}\}$$ has positive $\bP$ probability.  It follows from \Cref{lem:loop_bernoulli} that conditioned on $|\cL_{f,f'}| \ge k$, $\Delta h(f) - \Delta h(f')$ is equal in law to $\sum_{L \in \cL(f,f')} \xi_L$ and consequently we obtain $$\bP(\cD(f,f',k) \cap \Delta h(f) - \Delta h(f') \ge k/2) >0.$$ 
Since on the event above, both $f,f'$ are adjacent to some bi-infinite path, and it is a simple observation that $\Delta h$ is constant along the faces which are adjacent to a path, we obtain that $\Delta h$ has an infinite component with all $\Delta h$ values at least $k/2$. On each such component, at least one of the height functions must have height at least $k/4$. This is a contradiction to \Cref{cor:ind_height_perc} for a large choice of $k$.

\textbf{Case 2:} $\P(\cD_{\text{fin}})>0$.
On this event,  pick one bi-infinite path uniformly, call it $P$. Clearly $P$ partitions the faces into two subsets, color them uniformly red and blue. Let ${\sf C}:= (C(f))_{f \in F(G)}$ denote the color of the faces. It is straightforward to verify that $(\Gamma, \rho, M^S_X, M^T_X, \sf C )$ is a stationary random rooted marked graph.
Pick an arbitrary choice of the circle packing of $(\Gamma, \rho)$. Run 4 copies of independent simple random walks, $(X^i_n)_{n \ge 0}$ for $i \in \{1,2,3,4\}$ started from $\rho$. Let  
$R(\eta)$ be the event that the random walks converge to 4 distinct points on the boundary, and the minimal distance between each limit point is at least $\eta$.
Let $\cG(\eta,N)$ be the event that for all $n \ge N $ for each $i \in \{1,2,3,4\}$, 
\begin{align*}
&|\{0 \le j \le n:  X^i_j \text{ has a blue face adjacent to it}\}| \ge \eta n \text{, and, }\\
&|\{0 \le j \le n:  X^i_j \text{ has a red face adjacent to it}\}| \ge \eta n.
\end{align*}


\begin{lemma}\label{lem:G}
For all $\eps>0$ there exists a $\eta_0>0$ and $N_0\ge 1$ such that for all $\eta<\eta_0$ and $N \ge N_0$, $\P(\cG(\eta, N) \cap R(\eta)) \ge 1-\eps$.
\end{lemma}

\begin{proof}
By the nonatomicity of the exit measure (\Cref{thm:RW_circle_packing}), we can choose $\eta>0$ small enough so that the first item holds with probability at least $1-\ve/3$. Since the probability that $\rho$ is adjacent to a blue face or a red face are both positive, by the Ergodic theorem, the density of times when the random walker is adjacent to a red face is strictly positive and the same is true for a blue face. Consequently, we can choose a further smaller $\eta$ if needed, and a large enough $N$ to ensure that the second event holds with probability at least $1-\ve/3$.
\end{proof}

Using \Cref{lem:G}, choose $\eta, N$ so that $\P(\cG(\eta, N) \cap R(\eta)) \ge 1/2$. Now given such an $\eta$, using \Cref{lem:correlation}, choose $\delta>0$ small so that for all $f_1,f_2,f_3,f_4 \not \in (1-\delta)\D$ and $|f_i - f_j| > \eta$ for all $i \neq j$,
\begin{equation}
|\E((\Delta h(f_1) - \Delta h(f_2))(\Delta h (f_3) - \Delta h(f_4)) | < \eta^{10}. \label{eq:correlation}
\end{equation}
Let $J^i_N$ be a uniformly picked index from $\{1,\ldots, N\}$. Using the convergence of random walk to $\partial \D$, we can increase $N$ if needed to ensure that 
\begin{equation}
\P(X_{J^i_N} \in (1-\delta )\D \text{ for some }i\in \{1,2,3,4\}) < \frac{\eta^{10}}{4M^2(2)}.\label{eq:uniform_point}
\end{equation}
where $M(2)$ is as in \Cref{cor:height_moment}

With the quantifiers set up as above, we are ready to arrive at the desired contradiction. Fix $\eps=1/2$, and then fix $\eta, N$ as above and then run 4 independent random walks until step $N$, pick independently $J^i_N$ as above, and then for each $i$ pick a face $F_i$ uniformly at random from those adjacent to $X_{J^i_N}$. Using \eqref{eq:correlation}, \eqref{eq:uniform_point}, we conclude that 
\begin{equation}
\E((\Delta h(F_1) - \Delta h(F_2))(\Delta h (F_3) - \Delta h(F_4))  < 2\eta^{10}.\label{eq:correlation2}
\end{equation}

Now using \Cref{lem:loop_bernoulli} we will prove a contradictory lower bound of the left hand side of \eqref{eq:correlation2}.
Condition on $(\Gamma, \rho)$ and $M^S_X \Delta M^T_X$ and the random walk trajectories $X^i$, $i \in \{1,2,3,4\}$ and assume $\cF$ is the sigma algebra generated by them. Using \Cref{lem:loop_bernoulli}, we conclude that
\begin{align*}
\E\left(\Delta h(F_1) - \Delta h (F_2))(\Delta h(F_3) - \Delta h (F_4)\right) & = \E\left(\left(\sum_{L \in \cL(F_1,F_2)}\xi_{L}\right)\left(\sum_{L \in \cL(F_3,F_4)}\xi_{L}\right)\right) \\
& = \E(|\cL(F_1,F_2) \cap \cL(F_3,F_4)|) \\
&\ge \E(\E(|\cL(F_1,F_2) \cap \cL(F_3,F_4)|| \cF)1_{\cG(\eta, N) \cap R(\eta)})\\
&\ge  \left(\frac{\eta}{\Delta}\right)^4 \P(\cG(\eta, N) \cap R(\eta))\\
& = \frac12\left(\frac{\eta}{\Delta}\right)^4.
\end{align*}
In the fourth line above, we use the fact that on $\cG(\eta, N)$ with probability at least $(\eta/\Delta)^4$, $F_1,F_3$ are red and $F_2,F_4$ are blue since the density of indices which are red or blue are at least $\eta$ by definition. In which case, $|\cL(F_1,F_2) \cap \cL(F_3,F_4)| \ge 1$ since it must contain $P$.
This contradicts \Cref{eq:correlation} for a small enough choice of $\eta$.
\end{proof}
We now finish with the proof of \Cref{thm:main}.
\begin{proof}[Proof of \Cref{thm:main}]
Let $\chi \subset \partial \D$ be the set of $x \in \partial \D$ for which \Cref{thm:main_random} is true for $\mu^x$. It follows from \Cref{thm:main_random} that $\chi$ must have exit measure 1 for a simple random walk started at $\rho$, and it follows from the full support of the exit measure (\Cref{thm:RW_circle_packing}) that $\chi$ must be dense.
\end{proof}
\section{Extensions and open questions.}\label{sec:open}
We believe that the bounded degree assumption or nonamenability of the triangulation are not necessary, and the result should hold as long as the triangulation is CP hyperbolic.

\begin{conjecture}
\Cref{thm:main0,thm:main} hold for CP hyperbolic triangulations.
\end{conjecture}
Our results should extend beyond triangulations to graphs whose duals have finite degree. There are embeddings other than circle packing which should prove useful in that regard, examples of such embeddings include the square tiling, see \cite{sq_tiling_BS,sq_tiling_tom}.

\bibliographystyle{abbrv}
\bibliography{DD_hyperbolic}
\end{document}